\documentclass[]{aspm}

\articleinfo{}{}{}

\usepackage{verbatim}
\usepackage{amssymb}
\usepackage{amsbsy}
\usepackage{amscd}
\usepackage{amsmath}
\usepackage{amsthm}
\usepackage[mathscr]{eucal}
\usepackage{mathtools}
\usepackage{enumerate}

\newtheorem{thm}{Theorem}[section]
\newtheorem{conj}[thm]{Conjecture}
\newtheorem{cor}[thm]{Corollary}
\newtheorem{lem}[thm]{Lemma}

\newtheorem{dfn}[thm]{Definition}

\newtheorem{prop}[thm]{Proposition}

\newtheorem{property}{Property}

\def\dim{\mathop{\mathrm{dim}}\nolimits}

\def\det{\mathop{\mathrm{det}}\nolimits}

\def\wt{\mathop{\mathrm{wt}}\nolimits}
\def\rk{\mathop{\mathrm{rk}}\nolimits}

\newcommand{\mf}[1]{{\mathfrak{#1}}}
\newcommand{\mb}[1]{{\mathbf{#1}}}
\newcommand{\bb}[1]{{\mathbb{#1}}}
\newcommand{\mca}[1]{{\mathcal{#1}}}
\newcommand{\mr}[1]{{\mathrm{#1}}}

\newcommand{\da}[1]{\big\downarrow\raise.5ex\rlap{$\scriptstyle#1$}}

\title[Non-toric examples of elliptic canonical bases I]{Non-toric examples of elliptic canonical bases I}

\author[T. Hikita]{Tatsuyuki Hikita}

\address{\textsc{Research Institute for Mathematical Sciences, Oiwake Kita-Shirakawa Sakyo Kyoto 606-8502 JAPAN}}

\email{thikita@kurims.kyoto-u.ac.jp}

\rcvdate{}
\rvsdate{}



\begin{document}

\begin{abstract}
We propose several properties of elliptic lifts of the $K$-theoretic canonical bases for conical symplectic resolutions defined in our previous work. As an example, we construct elliptic lifts of canonical bases for the Hilbert scheme of 2-points in the affine plane.
\end{abstract}

\maketitle

\section{Introduction}

Canonical bases are fundamental object of study in representation theory. Lusztig (\cite{Lu1}, \cite{Lu2}) defined the canonical bases in equivariant $K$-theory of Springer resolutions or Slodowy varieties and conjectured that they control the modular representation theory of semisimple Lie algebras in positive characteristic. In the proof of Lusztig's conjecture for large enough characteristic by Bezrukavnikov-Mirkovi\'c \cite{BM}, canonical bases are understood to be indecomposable summands of certain tilting bundles and hence also have algebro-geometric applications such as \cite{ABM}. In our previous paper \cite{H1}, we defined canonical bases in equivariant $K$-theory of conical symplectic resolutions which are equipped with a symplectic dual \cite{BLPW} by using $K$-theoretic stable bases defined by Okounkov \cite{O1} as standard bases, and stated several conjectures generalizing Lusztig's conjecture to conical symplectic resolutions.

In \cite{H1}, we also initiated a study of elliptic analogue of the canonical bases for conical symplectic resolution. More precisely, we defined elliptic analogue of the bar involutions associated with dual pairs of conical symplectic resolutions and constructed certain nice families of elliptic bar invariant objects which reduces to $K$-theoretic canonical bases under certain limits in the toric cases. The aim of this paper is to provide non-toric examples of elliptic canonical bases in order to justify our approach to the study elliptic canonical bases. 

Our primary goal of this project is to give a characterization of what should be called elliptic canonical bases in some generality. Toric examples tell us many important properties which should be satisfied by the elliptic canonical bases. One of the interesting properties of them is that there are certain duality under symplectic duality or 3d-mirror symmetry which refines the duality of elliptic stable bases studied for example by \cite{RSVZ}. This gives us certain bilinear relations between elliptic canonical bases for dual pairs of symplectic resolutions. It turns out that this property is also fundamental in proving the elliptic bar-invariance of elliptic canonical bases. 

The second main property is that their $K$-theory limits give $K$-theoretic canonical bases for the corresponding slopes defined in \cite{H1}. More precisely, we define $K$-theoretic canonical bases for not necessarily generic slopes by using the $K$-theoretic stable bases for any slopes studied by Kononov-Smirnov \cite{KS}, which might be of independent interest. Since we have conjectural Kazhdan-Lusztig type algorithm to compute $K$-theoretic canonical bases, this property will give us leading terms of the elliptic canonical bases. 

The third main property is that they satisfy simple $q$-difference equations under the $q$-shifts of K\"ahler parameters. The explicit forms of the $q$-difference equations can be determined by looking at the leading terms determined by computing $K$-theory limits. This property comes from the property of $K$-theoretic canonical bases that tensor products of $K$-theoretic canonical bases and line bundles are also $K$-theoretic canonical bases for another slope. These $q$-difference equations have finite dimensional holomorphic solutions, and hence constrain the dependence on K\"ahler parameters of elliptic canonical bases. By using the symplectic duality and the above bilinear relations, one can also find $q$-difference equations for equivariant parameters. 

In this paper, we construct the solutions of the above bilinear relations under the above $q$-difference equations in the case of the Hilbert scheme of 2-points in the affine plane. It turns out that we have 3 parameters depending on the equivariant parameter for the conical $\bb{C}^{\times}$-action. In the forthcoming paper \cite{H2}, we will also accomplish this kind of calculations in the case of $T^{\ast}\mr{Gr}(2,4)$ and its symplectic dual. The remaining parameters are essentially the choice of a normalization factor $\Upsilon(v;q)$, and should be fixed by some other constraints which are not known to the author yet. 

The contents of this paper is as follows. In section 2, we summarize expected properties of elliptic canonical bases. We also state a conjecture characterizing $K$-theoretic canonical bases for any slopes. In section 3, we compute $K$-theoretic canonical bases for the Hilbert scheme of 2-points in the affine plane. In section 4, we solve the above bilinear relations in this case and give a brief speculation on a possible relation to the theory of vertex operator superalgebras (VOSA).

\subsection*{Acknowledgment}

The author thanks Sven M\"oller for valuable discussions related to this work, and Hitoshi Konno for invitation to the workshop ``Elliptic Integrable Systems, Representation Theory and Hypergeometric Function''. This work was supported by JSPS KAKENHI Grant Number 17K14163, 21H04993, and 21H04429.

\section{Expected properties of elliptic canonical bases}

In this section, we discuss what properties should be satisfied by elliptic canonical bases. We try to state them in a general setting so that it could be used in the future, although the only example we provide here is one of the easiest cases except for toric ones. It is possible that we might need further modifications at some points in the general cases. 

\subsection{Dual pairs}

We first recall the setting of symplectic duality \cite{BLPW} or its handy variant simply called dual pair in \cite[Definition 3.2]{H1}. We only explain what we need in this paper, and we refer to \cite{H1} for more detail. 

Let $X$ be a conical symplectic resolution, i.e., $X$ is a smooth connected algebraic variety over $\bb{C}$ equipped with a conical $\bb{S}\coloneqq\bb{C}^{\times}$-action and an algebraic symplectic form of $\bb{S}$-weight 2 satisfying:
\begin{itemize}
	\item $\bb{C}[X]$ is nonnegatively graded with respect to $\bb{S}$ and $\bb{C}[X]^{\bb{S}}=\bb{C}$,
	\item $X\rightarrow\mr{Spec}(\bb{C}[X])$ is a resolution of singularity.
\end{itemize}
We assume that there exists a Hamiltonian action of a torus $H$ such that it commutes with the $\bb{S}$-action and $\dim X^H=0$. We choose a generic cocharacter $\xi\in\bb{X}_{\ast}(H)$ of $H$ such that $X^{\xi(\bb{C}^{\times})}=X^{H}$. This determines the decomposition of the tangent fibers $T_{X}|_p=N_{p,+}\oplus N_{p,-}$ at each fixed point $p\in X^H$ into the attracting and repelling parts with respect to $\xi$. We fix $T^{1/2}_{X}\in K_{H\times\bb{S}}(X)$ called polarization of $X$ satisfying $T_X=T^{1/2}_{X}+v^{-2}\left(T^{1/2}_X\right)^{\vee}$ in the $H\times\bb{S}$-equivariant $K$-theory $K_{H\times\bb{S}}(X)$. Here, $v\in K_{\bb{S}}(\mr{pt})$ is the equivariant parameter for $\bb{S}$ and $\mca{E}^{\vee}$ means the dual of $\mca{E}\in K_{H\times\bb{S}}(X)$. 

We also fix another torus $K$ called K\"ahler torus which acts trivially on $X$ and equipped with a homomorphism $\mca{L}:\bb{X}_{\ast}(K)\rightarrow\mr{Pic}^{H\times\bb{S}}(X)$ to $H\times\bb{S}$-equivariant Picard group of $X$. We assume that there exists $\eta\in\bb{X}_{\ast}(K)$ such that $\mca{L}(\eta)$ is an ample line bundle on $X$. For each character $\alpha\in\bb{X}^{\ast}(H)$ or $\beta\in\bb{X}^{\ast}(K)$, we will write $a^{\alpha}\in K_{H}(\mr{pt})$ or $z^{\beta}\in K_{K}(\mr{pt})$ for the corresponding elements in equivariant $K$-theory. We extend $\mca{L}$ to $\mca{L}:\bb{X}_{\ast}(K)\times\bb{X}^{\ast}(H)\rightarrow\mr{Pic}^{H\times\bb{S}}(X)$ by $\mca{L}(\lambda,\alpha)=a^{\alpha}\mca{L}(\lambda)$ and assume that there exists $\kappa\in\bb{X}_{\ast}(K)\times\bb{X}^{\ast}(H)$ such that $\det T^{1/2}_X\cong\mca{L}(\kappa)$ in $\mr{Pic}^H(X)$. 

We will denote by $\mf{X}=(X,H,K,\mca{L},\xi,\eta,\kappa)$ the conical symplectic resolution $X$ equipped with the above data. 

\begin{dfn}\label{dual_pair}
We say that $\mf{X}$ and $\mf{X}^!=(X^!,H^!,K^!,\mca{L}^!,\xi^!,\eta^!,\kappa^!)$ form a dual pair if they are equipped with identifications $X^H\xrightarrow{\sim}(X^!)^{H^!}$ denoted by $p\mapsto p^!$, $H\cong K^!$, and $K\cong H^!$ satisfying the following:
\begin{itemize}
	\item We have $\xi=\eta^!$ and $\eta=\xi^!$ under the above identification.
	\item For any $p\in X^H$, $\lambda\in\bb{X}_{\ast}(K)$, and $\lambda^!\in\bb{X}_{\ast}(K^!)$, we have 
	\begin{align*}
		\langle\wt_H\mca{L}(\lambda)|_p,\lambda^!\rangle&=-\langle\wt_{H^!}\mca{L}^!(\lambda^!)|_{p^!},\lambda\rangle,\\
		\wt_{\bb{S}}\,\mca{L}(\lambda)|_p&=-\langle\wt_{H^!}\det N_{p^!,-},\lambda\rangle,\\
		\wt_{\bb{S}}\,\mca{L}^!(\lambda^!)|_{p^!}&=-\langle\wt_{H}\det N_{p,-},\lambda^!\rangle.
	\end{align*}
	\item We have $\wt_{\bb{S}}\det N_{p,-}+\frac{1}{2}\dim X=-\left(\wt_{\bb{S}}\det N_{p^!,-}+\frac{1}{2}\dim X^!\right)$.
	\item For any $v^ma^{\alpha}$ appearing in $N_{p,-}$, we have $m\equiv\langle \alpha,\kappa^!\rangle\mod2$.
	\item For any $v^mz^{\beta}$ appearing in $N_{p^!,-}$, we have $m\equiv\langle \beta,\kappa\rangle\mod2$.
\end{itemize}
\end{dfn}

We remark that the last two new conditions are slightly refined version of \cite[Assumption 3.11]{H1}, i.e., these conditions imply 
\begin{align*}
	\mr{wt}_{\bb{S}}(\det T^{1/2}_X|_{p})-\mr{wt}_{\bb{S}}(\mca{L}(\kappa)|_{p})\equiv\frac{1}{2}\dim X+\frac{1}{2}\dim X^!\mod2.
\end{align*}
We will use the following shorthand notations.

\begin{dfn}
For $\mf{X}=(X,H,K,\mca{L},\xi,\eta,\kappa)$, we define its opposite by $-\mf{X}=(X,H,K,\mca{L},-\xi,\eta,\kappa)$. If a dual of $-\mf{X}^!$ exists, we denote it by $\mf{X}_{\mr{flop}}=(X_{\mr{flop}},H,K,\mca{L}_{\mr{flop}},\xi,-\eta,\kappa)$ and call it a maximal flop of $\mf{X}$.
\end{dfn}

We note that we have an identification $X^H\cong(X^!)^{H^!}\cong(X_{\mr{flop}})^H$. The fixed point of $(X_{\mr{flop}})^H$ corresponding to $p\in X^H$ is denoted by $p_{\mr{flop}}$.

\subsection{Elliptic bar involutions}

We briefly recall the notion of elliptic bar involutions defined in \cite[Definition 4.14]{H1} which are based on the notion of elliptic stable bases introduced by Aganagic-Okounkov \cite{AO}. We do not recall the full definition, but recall some basic properties. For more details, see \cite{AO} and \cite{H1}.

Let $\mf{D}^{\circ}\coloneqq\{q\in\bb{C}\mid 0<|q|<1\}$ be the punctured unit disk and $F$ be the field of multi-valued meromorphic functions on $H\times K\times\bb{S}\times\mf{D}^{\circ}$. We define the theta function $\vartheta(x)=\vartheta(x;q)$ by 
\begin{align*}
	\vartheta(x)\coloneqq(x^{1/2}-x^{-1/2})\prod_{m\geq1}(1-q^mx)(1-q^mx^{-1})
\end{align*}
and write 
\begin{align*}
\vartheta\left(\sum_{i}b_i-\sum_{j}c_j\right)=\frac{\prod_i\vartheta(b_i)}{\prod_j\vartheta(c_j)}
\end{align*}
for characters $b_i,c_j\in\bb{X}^{\ast}(H\times K\times \bb{S})$. For $(\lambda^!,\lambda,m)\in\bb{X}_{\ast}(H\times K\times \bb{S})\otimes_{\bb{Z}}\bb{R}$, we define $q$-difference operators $\delta^{(\lambda^!,\lambda,m)}$ on $F$ by $a^{\alpha}\mapsto q^{\langle\lambda
^!,\alpha\rangle}a^{\alpha}$, $z^{\beta}\mapsto q^{\langle\lambda,\beta\rangle}z^{\beta}$, and $v\mapsto q^{m} v$ for any $\alpha\in\bb{X}^{\ast}(H)$, $\beta\in\bb{X}^{\ast}(K)$. We will write $\delta_a^{\lambda^!}\coloneqq\delta^{(\lambda^!,0,0)}$, $\delta_z^{\lambda}\coloneqq\delta^{(0,\lambda,0)}$, and $\delta^m_{v}\coloneqq\delta^{(0,0,m)}$.

The elliptic stable bases are certain sections of some line bundles on the $H\times K\times \bb{S}$-equivariant elliptic cohomology defined by using the elliptic curve $\bb{C}^{\times}/q^{\bb{Z}}$. In order to freely manipulate them, we will consider them as a tuple of meromorphic functions by fixed point restrictions. More precisely, we define 
\begin{align*}
\mb{E}(\mf{X})_{\mr{loc}}\coloneqq\oplus_{p\in X^{H}}F
\end{align*}
and consider the renormalized elliptic stable bases $\mr{Stab}^{ell}_{\mf{X}}(p)$ as an element of $\mb{E}(\mf{X})_{\mr{loc}}$ by $\mr{Stab}^{ell}_{\mf{X}}(p)=(\mr{Stab}^{ell}_{\mf{X}}(p)|_{p'})_{p'\in X^H}$. Our normalization for $\mr{Stab}^{ell}_{\mf{X}}(p)$ is explained in \cite[Definition 4.2]{H1}. For example, we have 
\begin{align*}
	\mr{Stab}^{ell}_{\mf{X}}(p)|_{p}=\vartheta(N_{p,-})\vartheta(N_{p^!,-})
\end{align*}
and satisfies the following $q$-difference equations
\begin{align*}
	\delta^{\lambda^!}_a\left(\frac{\mr{Stab}^{ell}_{\mf{X}}(p_2)|_{p_1}}{\vartheta(N_{p_1,-})\vartheta(N_{p^!_2,-})}\right)&=\frac{\mca{L}^!(\lambda^!)|_{p_1^!}}{\mca{L}^!(\lambda^!)|_{p_2^!}}\cdot\frac{\mr{Stab}^{ell}_{\mf{X}}(p_2)|_{p_1}}{\vartheta(N_{p_1,-})\vartheta(N_{p^!_2,-})},\\
	\delta^{\lambda}_z\left(\frac{\mr{Stab}^{ell}_{\mf{X}}(p_2)|_{p_1}}{\vartheta(N_{p_1,-})\vartheta(N_{p^!_2,-})}\right)&=\frac{\mca{L}(\lambda)|_{p_2}}{\mca{L}(\lambda)|_{p_1}}\cdot\frac{\mr{Stab}^{ell}_{\mf{X}}(p_2)|_{p_1}}{\vartheta(N_{p_1,-})\vartheta(N_{p^!_2,-})},\\
	\delta^{m}_v\left(\frac{\mr{Stab}^{ell}_{\mf{X}}(p_2)|_{p_1}}{\vartheta(N_{p_1,-})\vartheta(N_{p^!_2,-})}\right)&=\delta^{m/2}_v\left(\frac{\det N_{p_2,-}}{\det N_{p_1,-}}\frac{\det N_{p_1^!,-}}{\det N_{p_2^!,-}}\right)^m\cdot\frac{\mr{Stab}^{ell}_{\mf{X}}(p_2)|_{p_1}}{\vartheta(N_{p_1,-})\vartheta(N_{p^!_2,-})}.
\end{align*}
for each $(\lambda^!,\lambda,m)\in\bb{X}_{\ast}(H\times K\times \bb{S})$ and $p_1,p_2\in X^{H}$. Moreover, 
\begin{align*}
	\sqrt{\mca{L}(-\kappa)|_{p_1}\mca{L}^!(-\kappa^!)|_{p^!_{2}}}\cdot\mr{Stab}^{ell}_{\mf{X}}(p_2)|_{p_1}
\end{align*}
is single-valued, i.e., only contains integral powers of $a$, $z$, $v$, and $q$. These properties together with certain support condition uniquely determine $\mr{Stab}^{ell}_{\mf{X}}(p)$. Since the matrix $\mr{Stab}^{ell}_{\mf{X}}\coloneqq(\mr{Stab}^{ell}_{\mf{X}}(p_2)|_{p_1})_{p_1,p_2\in X^H}$ is triangular with respect to certain order on $X^H$, the elliptic stable basis $\{\mr{Stab}^{ell}_{\mf{X}}(p)\}_{p\in X^H}$ forms an $F$-basis of $\mb{E}(\mf{X})_{\mr{loc}}$.

\begin{dfn}\label{Dfn_ellbar}
We call the $F$-semilinear map $\beta^{ell}_{\mf{X}}:\mb{E}(\mf{X})_{\mr{loc}}\rightarrow \mb{E}(\mf{X})_{\mr{loc}}$ with respect to the involution $f=f(a,z,v;q)\mapsto \bar{f}\coloneqq f(a,z,v^{-1};q)$ of $F$ which is characterized by 
\begin{align*}
	\beta^{ell}_{\mf{X}}\left(\mr{Stab}^{ell}_{\mf{X}}(p)\right)=(-1)^{\frac{\dim X}{2}}\, \mr{Stab}^{ell}_{-\mf{X}}(p)
\end{align*}
for any $p\in X^H$ the elliptic bar involution for $\mf{X}$.
\end{dfn}

We note that we slightly changed the normalization for elliptic bar involutions from \cite[Definition 4.14]{H1} by conjugation with $\sqrt{\mca{L}(\kappa)}$. Elliptic canonical bases should be an $F$-basis of $\mb{E}(\mf{X})_{\mr{loc}}$ which is invariant under $\beta^{ell}_{\mf{X}}$ and satisfies some good properties. Existence of elliptic canonical bases would imply that $\beta^{ell}_{\mf{X}}$ is an involution, but this is a nontrivial conjecture in general. 

\subsection{$K$-theory limits}

In this section, we define $K$-theoretic canonical bases for any slope $s\in\bb{X}_{\ast}(K)\otimes_{\bb{Z}}\bb{R}$. When $s$ is generic, i.e., $\langle s,\beta\rangle\notin\bb{Z}$ for any $\beta\in\bb{X}^{\ast}(K)$ appearing in the tangent weights of $X^!$ at $K$-fixed points, this essentially reduces to the original definition of \cite{H1}. 

We first recall $K$-theory limits of elliptic stable bases. We define 
\begin{align*}
	\mr{Stab}^{K}_{\mf{X},s}(p)\coloneqq(-v^{-1/2})^{\frac{\dim X}{2}}\,\lim_{q\rightarrow0}\delta^{-s}_{z}\left(\sqrt{\mca{L}(-\kappa)}\otimes\frac{\mr{Stab}^{ell}_{\mf{X}}(p)}{\vartheta(v N_{p^!,-})}\right)
\end{align*}
for each $p\in X^H$, where $\otimes$ means component-wise multiplication. When $s$ is generic, this coincides with the renormalized $K$-theoretic stable bases in \cite[Definition 3.12]{H1} since 
\begin{align*}
	\lim_{q\rightarrow0}\delta^{-s}_{z}\left(\frac{\vartheta(N_{p^!,-})}{\vartheta(vN_{p^!,-})}\right)=v^{\sum_{\beta}\left(\lfloor\langle s,\beta\rangle\rfloor+\frac{1}{2}\right)}
\end{align*}
where the sum is over the multiset of $\beta\in\bb{X}^{\ast}(K)$ appearing in $N_{p^!,+}$. In particular, this does not depend on the K\"ahler parameters. They form a basis of $\mb{K}(\mf{X})_{\mr{loc}}\coloneqq K_{H\times K\times\bb{S}}(X)\otimes_{K_{H\times K\times\bb{S}}(\mr{pt})}\mr{Frac}(K_{H\times K\times\bb{S}}(\mr{pt}))$. 

\begin{dfn}
We call the semilinear map $\beta^{K}_{\mf{X},s}:\mb{K}(\mf{X})_{\mr{loc}}\rightarrow \mb{K}(\mf{X})_{\mr{loc}}$ characterized by 
	\begin{align*}
		\beta^{K}_{\mf{X},s}(\mr{Stab}^{K}_{\mf{X},s}(p))=(-v)^{\frac{\dim X}{2}}\,\mr{Stab}^{K}_{-\mf{X},s}(p)
	\end{align*}
for any $p\in X^{H}$ the $K$-theoretic bar involution for $\mf{X}$ at slope $s$.
\end{dfn}

\begin{dfn}\label{K-th_can}
Assume that $s$ is generic. A $K_{H\times\bb{S}}(\mr{pt})$-basis $\{\mca{E}^{K}_{\mf{X},s}(p)\}_{p\in X^H}$ of $K_{H\times\bb{S}}(X)$ is called the $K$-theoretic canonical basis at slope $s$ if
\begin{enumerate}
	\item $\beta^{K}_{\mf{X},s}(\mca{E}^{K}_{\mf{X},s}(p))=\mca{E}^{K}_{\mf{X},s}(p)$,
	\item if we expand 
\begin{align*}
	\mca{E}^{K}_{\mf{X},s}(p)=\sum_{p'\in X^H}f_{p,p'}(v,a)\,\mr{Stab}^K_{\mf{X},s}(p')
\end{align*} 
for some $f_{p,p'}(v,a)\in\mr{Frac}(K_{H\times\bb{S}}(\mr{pt}))$, then $\lim\limits_{v\rightarrow\infty}f_{p,p'}(v,a)=\delta_{p,p'}$.
\end{enumerate}	
\end{dfn}

The $K$-theoretic canonical basis is unique if it exist. In \cite[Proposition 3.21]{H1}, we proved the formal existence of $K$-theoretic canonical basis under the assumption that $\beta^K_{\mf{X},s}$ is an involution. One of our main conjecture in \cite[Conjecture 3.38]{H1} is that it is given by the classes of indecomposable summands of certain tilting bundle on $X$. For generic $s$, we set $\bb{B}^K_{\mf{X},s}\coloneqq\{a^{\alpha}\mca{E}^K_{\mf{X},s}(p)\mid \alpha\in\bb{X}^{\ast}(H),p\in X^H\}$ and denote by 
\begin{align*}
	\widetilde{\bb{B}}^K_{\mf{X}}\coloneqq\bigcup_{s:\mr{generic}}\bb{B}^K_{\mf{X},s}\subset K_{H\times\bb{S}}(X)
\end{align*}
the union of all $K$-theoretic canonical bases at generic slopes. Obviously, we have natural action of $\bb{X}^{\ast}(H)$ on $\bb{B}^K_{\mf{X},s}$. By \cite[Lemma 3.34]{H1}, we also have an action of $\bb{X}_{\ast}(K)$ on $\widetilde{\bb{B}}^K_{\mf{X}}$ by tensor product of line bundles of the form $\mca{L}(\lambda)$ for $\lambda\in\bb{X}_{\ast}(K)$.

When $s$ is not generic, we still expect the existence of $K$-theoretic canonical bases which depends on K\"ahler parameters. Let $\varepsilon$ be a sufficiently small positive real number and take generic slopes $s_{\pm}\coloneqq s\pm\varepsilon\eta$ near $s$. 

\begin{conj}\label{conj_wall_can}
For any $s\in\bb{X}_{\ast}(K)\otimes_{\bb{Z}}\bb{R}$	 and $p\in X^H$, there exists unique $\mca{E}^K_{\mf{X},s}(p)\in\mb{K}(\mf{X})_{\mr{loc}}$ of the form 
\begin{align*}
	\mca{E}^K_{\mf{X},s}(p)=\mca{E}^K_{\mf{X},s_+}(p)+\sum_{\substack{\beta\in\bb{X}^{\ast}(K)\\\langle s, \beta\rangle\in\bb{Z}\\0<\langle\eta,\beta\rangle<\langle\eta,\beta_{\mr{max}}\rangle}}z^{-\beta}\,\mca{F}_{\beta}\pm z^{-\beta_{\mr{max}}}\,\mb{wc}_{s_+,s_-}\left(\mca{E}^K_{\mf{X},s_+}(p)\right),
\end{align*}
where $\mb{wc}_{s_+,s_-}:\bb{B}^K_{\mf{X},s_+}\xrightarrow{\sim}\bb{B}^K_{\mf{X},s_-}$ is a bijection compatible with $\bb{X}^{\ast}(H)$-action such that
\begin{enumerate}
	\item $\beta^K_{\mf{X},s}\left(\mca{E}^K_{\mf{X},s}(p)\right)=\mca{E}^K_{\mf{X},s}(p)$,
	\item $\mca{F}_{\beta}\in \left(\sum\limits_{\mca{E}\in\bb{B}^K_{\mf{X},s_+}}v^{-1}\bb{Z}[v^{-1}]\,\mca{E}\right)\cap\left(\sum\limits_{\mca{E}\in\bb{B}^K_{\mf{X},s_-}}v\bb{Z}[v]\,\mca{E}\right)$.
	\item $\mb{wc}_{s_+,s_-}\left(\mca{E}^K_{\mf{X},s_+}(p)\right)\in \sum\limits_{\mca{E}\in\bb{B}^K_{\mf{X},s_+}}v^{-1}\bb{Z}[v^{-1}]\,\mca{E}$ if $\beta_{\mr{max}}\neq0$.
\end{enumerate}
\end{conj}

We note that the uniqueness in a formal completion 
\begin{align*}
\left\{\sum_{\beta\in\bb{X}^{\ast}(K)}\mca{G}_{\beta}z^{-\beta}\middle|\substack{\mca{G}_{\beta}\in K_{H\times\bb{S}}(X)\\\forall C>0,\#\{\beta\mid \mca{G}_{\beta}\neq0, \langle\beta,\eta\rangle<C\}<\infty} \right\}
\end{align*} 
follows from the existence of $\mca{E}^K_{\mf{X},s}(p)$ for any $p\in X^H$ by the same argument as in \cite[Proposition~3.21]{H1}. We expect that this can be calculated in principle by Kazhdan-Lusztig type algorithm just as in \cite[Conjecture 3.49]{H1}. We note that the bijection $\mb{wc}_{s_+,s_-}$ does not coincide in general with the bijection considered in \cite[Conjecture 3.49]{H1}. This happens for example when $X=\mr{Hilb}^n(\bb{C}^2)$ and $s=0$ as one can see from Proposition~\ref{prop_wall_K-th_can}.

We define equivalence relation $\sim$ on $\widetilde{\bb{B}}^K_{\mf{X}}$ generated by $\mca{E}\sim\mb{wc}_{s_+,s_-}(\mca{E})$ for any $\mca{E}\in\bb{B}^K_{\mf{X},s_+}$ and $s\in\bb{X}_{\ast}(K)\otimes_{\bb{Z}}\bb{R}$ as above and set $\Xi_{\mf{X}}\coloneqq\bb{X}^{\ast}(H)\backslash\widetilde{\bb{B}}^K_{\mf{X}}/\sim$ the set of equivalence classes modulo character twists. For any $\mu\in\Xi_{\mf{X}}$ and $s\in\bb{X}_{\ast}(K)\otimes_{\bb{Z}}\bb{R}$, there exists $p\in X^H$ such that $\mca{E}^K_{\mf{X},s_+}(p)\in\mu$ and we set $\left[\mca{E}^{K}_{\mf{X},s}(\mu)\right]\coloneqq\pm\bb{X}^{\ast}(H\times K)\cdot\mca{E}^K_{\mf{X},s}(p)$ for such $p\in X^H$. We expect that $\mca{E}\sim\mca{E}'$ for $\mca{E},\mca{E}'\in\bb{B}^K_{\mf{X},s}$ implies $\mca{E}=\mca{E}'$, and hence such $p$ should be uniquely determined. In particular, we would obtain a bijection $\iota_{\mf{X}}:\Xi_{\mf{X}}\xrightarrow{\sim} X^H$ by taking $s=0$, i.e., $\mca{E}^{K}_{\mf{X},\varepsilon\eta}(\iota_{\mf{X}}(\mu))\in\mu$. 

We will parametrize elliptic canonical basis by the set $\Xi_{\mf{X}}$ and denote by $\mca{E}^{ell}_{\mf{X}}(\mu)$ the elliptic canonical basis parametrized by $\mu\in\Xi_{\mf{X}}$. The first expected property of elliptic canonical basis is the holomorphicity and the behavior of its $K$-theory limits. Let $\widetilde{\mf{D}^{\circ}}$ be the double cover of $\mf{D}^{\circ}$.

\begin{property}
	For any $\mu\in\Xi_{\mf{X}}$, there exist $\alpha_{\mu}\in\bb{X}^{\ast}(H)\otimes_{\bb{Z}}\bb{Q}$, $\beta_{\mu}\in\bb{X}^{\ast}(K)\otimes_{\bb{Z}}\bb{Q}$, and $r_{\mu}\in\bb{Q}$ such that
	\begin{enumerate}
		\item each component of $q^{-r_{\mu}}a^{-\alpha_{\mu}}z^{-\beta_{\mu}}\sqrt{\mca{L}(-\kappa)}\otimes\mca{E}^{ell}_{\mf{X}}(\mu)$ is single-valued and holomorphic function on $H\times K\times\bb{S}\times\widetilde{\mf{D}^{\circ}}$,
		\item for any $s\in\bb{X}_{\ast}(K)\otimes_{\bb{Z}}\bb{R}$, there exists $r_{\mu}(s)\in\bb{R}$ such that 
\begin{align*}
	\lim_{q\rightarrow0}q^{-r_{\mu}(s)}a^{-\alpha_{\mu}}z^{-\beta_{\mu}}\sqrt{\mca{L}(-\kappa)}\otimes\delta^{-s}_{z}\left(\mca{E}^{ell}_{\mf{X}}(\mu)\right)\in\left[\mca{E}^K_{\mf{X},s}(\mu)\right].
\end{align*}
	\end{enumerate}
\end{property}

In the case of toric hyper-K\"ahler manifolds, this property essentially follows from \cite[Proposition 6.13]{H1}. In view of possible relations to VOSA, we use the double cover $\widetilde{\mf{D}^{\circ}}$ in order to allow half-integral powers of $q$ since characters of VOSA often have such half-integral powers.

\subsection{Duality of elliptic canonical bases}

Our main proposal in this paper is to replace the elliptic bar invariance by a certain duality between elliptic canonical bases for $\mf{X}$ and $\mf{X}^!$. We first recall the expected duality of elliptic stable bases under symplectic duality studied for example by \cite{RSVZ}. For $p\in X^H$, we denote by $\mr{ind}_{p}$ (resp. $\mr{ind}_{p^!}$) the attracting part of $\left.T^{1/2}_X\right|_{p}$ (resp. $\left.T^{1/2}_{X^!}\right|_{p^!}$) with respect to $\xi\in\bb{X}_{\ast}(H)$ (resp. $\eta\in\bb{X}_{\ast}(K)$) and set 
\begin{align*}
	\sigma_{\mf{X},\mf{X}^!}(p)\coloneqq(-1)^{\rk\mr{ind}_p+\rk\mr{ind}_{p^!}}.
\end{align*}

\begin{conj}
For any $p_1,p_2\in X^H$,  we have
\begin{align*}
	\sigma_{\mf{X},\mf{X}^!}(p_1)\left.\mr{Stab}^{ell}_{\mf{X}}(p_1)\right|_{p_2}=\sigma_{\mf{X}^!,\mf{X}}(p_2^!)\left.\mr{Stab}^{ell}_{\mf{X}^!}(p_2^!)\right|_{p_1^!}.
\end{align*}	
\end{conj}

We note that we have a bijection $\Xi_{\mf{X}}\xrightarrow{\iota_{\mf{X}}} X^H\cong(X^!)^{H^!}\xrightarrow{\iota_{\mf{X}^!}^{-1}}\Xi_{\mf{X}^!}$ which we denote by $\mu\mapsto\mu^!$ for $\mu\in\Xi_{\mf{X}}$. Expected duality of elliptic canonical bases refines this duality.

\begin{property}
There exists a (possibly multi-valued) holomorphic function $\Upsilon(v;q)$ on $\bb{S}\times\mf{D}^{\circ}$ satisfying $\Upsilon(v^{-1};q)=\Upsilon(v;q)$ such that for any $p_1,p_2\in X^H$, we have
\begin{align*}
	\sigma_{\mf{X},\mf{X}^!}(p_1)\Upsilon(v;q)\cdot\left.\mr{Stab}^{ell}_{\mf{X}}(p_1)\right|_{p_2}=\sum_{\mu\in\Xi_{\mf{X}}}\left.\mca{E}^{ell}_{\mf{X}^!}(\mu^!)\right|_{p_1^!}\cdot\left.\mca{E}^{ell}_{\mf{X}}(\mu)\right|_{p_2}.
\end{align*}
\end{property}

In the case of toric hyper-K\"ahler manifolds, we do not need the factor $\Upsilon(v;q)$ by \cite[Theorem 6.9]{H1}. We allow this freedom here because we do not know how to fix the normalization of elliptic canonical basis yet.

By applying Property B to the dual pair $(-\mf{X},\mf{X}^!_{\mr{flop}})$, the elliptic bar invariance of $\mca{E}^{ell}_{\mf{X}}(\mu)$ easily follows from the next expected property of elliptic canonical basis.

\begin{property}
We have $\mca{E}^{ell}_{-\mf{X}}(\mu)=\mca{E}^{ell}_{\mf{X}}(\mu)$ and $\left.\mca{E}^{ell}_{\mf{X}_{\mr{flop}}}(\mu)\right|_{p_{\mr{flop}}}=\overline{\left.\mca{E}^{ell}_{\mf{X}}(\mu)\right|_p}$ for any $\mu\in\Xi_{\mf{X}}$ and $p\in X^H$.
\end{property}

In the case of toric hyper-K\"ahler manifolds, this follows from the explicit formula in \cite[Definition 6.5]{H1}. 

\subsection{K\"ahler $q$-difference equations}

The final expected property stated in this paper is that elliptic canonical basis should satisfy certain $q$-difference equations under $\delta^{\lambda}_{z}$. We note that $\mca{E}\sim\mca{E}'$ implies $\mca{L}(\lambda)\otimes\mca{E}\sim\mca{L}(\lambda)\otimes\mca{E}'$ for any $\lambda\in\bb{X}_{\ast}(K)$. For $\lambda\in\bb{X}_{\ast}(K)$ and $\mu\in\Xi_{\mf{X}}$, we denote by $\lambda+\mu\in\Xi_{\mf{X}}$ the class of $\mca{L}(\lambda)\otimes\mca{E}$ for $\mca{E}\in\mu$. 

\begin{property}
For any $\mu\in\Xi_{\mf{X}}$, there exists a linear map $G_{\mu}:\bb{X}_{\ast}(K)\otimes_{\bb{Z}}\bb{Q}\rightarrow\bb{X}^{\ast}(K)\otimes_{\bb{Z}}\bb{Q}$ such that for any $\lambda\in\bb{X}_{\ast}(K)$, we have
\begin{align*}
	\delta^{\lambda}_{z}\left(\mca{E}^{ell}_{\mf{X}}(\mu)\right)=\pm q^{-\frac{\langle\lambda,G_{\mu}\lambda\rangle}{2}}z^{-G_\mu\lambda}\mca{L}(-\lambda)\otimes\mca{E}^{ell}_{\mf{X}}(\lambda+\mu).
\end{align*} 
\end{property}

In the case of toric hyper-K\"ahler manifold, this also follows from the explicit formula in \cite[Definition 6.5]{H1}. The exact form of the $q$-difference equations can be extracted from the information of various leading terms of $\mca{E}^{ell}_{\mf{X}}(\mu)$ in Property~A. These $q$-difference equations and holomorphicity of $\mca{E}^{ell}_{\mf{X}}(\mu)$ implies that $z$-dependence of $\mca{E}^{ell}_{\mf{X}}(\mu)$ can be written as a finite linear combination of theta functions associated with some lattices. By using Property~B, one can also extract $q$-difference equations satisfied by each component of $\mca{E}^{ell}_{\mf{X}}(\mu)$ under $\delta^{\lambda^!}_{a}$. 

For $q$-difference equations under $\delta_{v}$, we also expect the following independence on $\mu\in\Xi_{\mf{X}}$.

\begin{property}
For each $p\in X^H$, there exists $x_p\in\pm q^{\bb{Q}}\cdot\bb{X}^{\ast}(H\times K\times\bb{S})$ such that for any $\mu\in\Xi_{\mf{X}}$, we have
\begin{align*}
	\delta_v\left(\left.\mca{E}^{ell}_{\mf{X}}(\mu)\right|_p\right)=x_p\cdot\left.\mca{E}^{ell}_{\mf{X}}(\mu)\right|_p.
\end{align*}
\end{property}

One can check this property in the case of toric hyper-K\"ahler manifolds by direct calculations using explicit formulas.

\subsection{Main conjecture}

Now we state our tentative conjecture about elliptic canonical bases. 

\begin{conj}
There exist $\mca{E}^{ell}_{\mf{X}}(\mu)$ for each $\mu\in\Xi_{\mf{X}}$ satisfying Property~A,B,C,D, and E.
\end{conj}

We should remark that these properties does not uniquely determine elliptic canonical basis. We naively speculate that some kind of modularity or quasi-modularity might determine the $q$-dependence completely, but the author could not find such conditions even in the simplest example we study in this paper. 

\section{$K$-theoretic canonical bases for Hilbert scheme of 2-points}

In this section, we compute $K$-theoretic canonical bases for Hilbert scheme of 2-points
\begin{align*}
	\mr{Hilb}^2(\bb{C}^2)\coloneqq\{\mca{I}\subset\bb{C}[x,y]\mbox{ : ideal}\mid\dim(\bb{C}[x,y]/\mca{I})=2\}.
\end{align*}
For details about the geometry of Hilbert scheme of points, see \cite{N}.  We actually consider the preimage $X\subset\mr{Hilb}^2(\bb{C}^2)$ at $(0,0)\in\bb{C}^2$ of the composition of Hilbert-Chow morphism $\mr{Hilb}^2(\bb{C}^2)\rightarrow\mr{Sym}^2(\bb{C}^2)$ and the sum $\mr{Sym}^2(\bb{C}^2)\rightarrow\bb{C}^2$. Since we have $\mr{Hilb}^2(\bb{C}^2)\cong X\times\bb{C}^2$, there are no essential differences. 

\subsection{Elliptic stable bases}

Let $H\coloneqq\bb{C}^{\times}_a$ be a torus acting on $\bb{C}^2$ by $(x,y)\mapsto(ax,a^{-1}y)$ and consider the induced action on $X$. The conical action of $\bb{S}\coloneqq\bb{C}^{\times}_v$ is induced by $(x,y)\mapsto(v^{-1}x,v^{-1}y)$. We have two $H$-fixed points labeled by $\{[2],[1,1]\}$ the partitions of 2, where $[2]=(x^2,y)$ and $[1,1]=(x,y^2)$. 

Let $\mca{V}$ be the tautological bundle on $X$ whose fiber at $\mca{I}$ is given by $\bb{C}[x,y]/\mca{I}$. This is $H\times\bb{S}$-equivariant and we have $\mca{V}|_{[2]}=1+va^{-1}$ and $\mca{V}|_{[1,1]}=1+va$. We set $K=\bb{C}^{\times}_z$ and 
\begin{align*}
	\mca{O}(1)\coloneqq v\det\mca{V}=\left(\begin{array}{c}
		v^2a^{-1}\\
		v^2a
	\end{array}\right)
\end{align*}
which is an ample line bundle on $X$. We define $\mca{L}:\bb{X}_{\ast}(K)\cong\bb{Z}\rightarrow\mr{Pic}^{H\times\bb{S}}(X)$ by $\mca{L}(m)=\mca{O}(1)^{\otimes m}$. We take a polarization 
\begin{align*}
T^{1/2}_X\coloneqq\mca{V}+(v^{-1}a-1)\,\mca{V}^{\vee}\otimes\mca{V}-v^{-1}a\,\mca{O}
\end{align*}
and $(\xi,\eta)=(1,1)\in\bb{X}_{\ast}(H\times K)\cong\bb{Z}^2$. In this choice, we have $N_{[2],-}=a^{-2}$, $N_{[1,1],-}=v^{-2}a^{-2}$, and $\det T^{1/2}_X\cong v^{-4}a^3\,\mca{O}(1)$. We take $\kappa=(1,3)\in\bb{X}_{\ast}(K)\times\bb{X}^{\ast}(H)$ and consider $\mf{X}=(X,H,K,\mca{L},\xi,\eta,\kappa)$. It is straightforward to check that $\mf{X}$ is self dual in the sense of Definition~\ref{dual_pair}.

\begin{lem}
$\mf{X}$ is self dual under the identifications $a^!=z$, $z^!=a$, and $X^H\cong(X^!)^{H^!}$ given by $[2]^!=[1,1]$ and $[1,1]^!=[2]$.	
\end{lem}

For a dual of $-\mf{X}$, we take $\mf{X}_{\mr{flop}}\coloneqq(X,H,K,\mca{L}_{\mr{flop}},\xi,-\eta,\kappa_{\mr{flop}})$ with $\mca{L}_{\mr{flop}}(\lambda)\coloneqq\mca{L}(-\lambda)$ and $\kappa_{\mr{flop}}=(-1,3)$. We have $\mca{L}_{\mr{flop}}(\lambda)|_{p_{\mr{flop}}}=\overline{\mca{L}(\lambda)|_p}$ by setting $[2]_{\mr{flop}}\coloneqq[1,1]$ and $[1,1]_{\mr{flop}}\coloneqq[2]$.

\begin{lem}
The pair $(-\mf{X},\mf{X}_{\mr{flop}})$ forms a dual pair under the identifications $a^!=z$, $z^!=a$, and $X^H\cong(X^!_{\mr{flop}})^{H^!}$ given by $[2]^!=[1,1]_{\mr{flop}}$ and $[1,1]^!=[2]_{\mr{flop}}$.
\end{lem}

A formula for elliptic stable basis of Hilbert scheme of points is given by Smirnov \cite{S}. In our situation, it is given as follows.

\begin{prop}\label{Hilb_ell_stab}
We have 
\begin{align*}
	\mr{Stab}^{ell}_{\mf{X}}([2])&=\left(\begin{array}{c}
		\vartheta(a^{-2})\vartheta(v^{-2}z^{-2})\\
		0
	\end{array}\right),\\
	\mr{Stab}^{ell}_{\mf{X}}([1,1])&=\left(\begin{array}{c}
		\frac{\vartheta(v^{-2})\left(\vartheta(a^{-2})\vartheta(vz^2a^{-1})\vartheta(v^{-1}z)+\vartheta(v^{-1}a^{-1})\vartheta(vza^{-2})\vartheta(z^{-2})\right)}{\vartheta(v^{-1}a)\vartheta(vz)}\\
		\vartheta(v^{-2}a^{-2})\vartheta(z^{-2})
	\end{array}\right).
\end{align*}
\end{prop}

\begin{proof}
This follows from \cite[Theorem~4]{S} by substituting $t_1=v^{-1}a$ and $t_2=v^{-1}a^{-1}$, shifting $z$ to $vz$ coming from the shift in \cite[Definition~4.2]{H1}, multiplying $\vartheta(N_{p^!,-})$, and then dividing by $\vartheta(v^{-1}a^{-1})$ which comes from the difference between $\mr{Hilb}^2(\bb{C}^2)$ and $X$.
\end{proof}

\begin{cor}
We have 
\begin{align*}
	\mr{Stab}^{ell}_{-\mf{X}}([2])&=\left(\begin{array}{c}
		\vartheta(v^{-2}a^{2})\vartheta(z^{-2})\\
		\frac{\vartheta(v^{-2})\left(\vartheta(a^{2})\vartheta(vz^2a)\vartheta(v^{-1}z)+\vartheta(v^{-1}a)\vartheta(vza^{2})\vartheta(z^{-2})\right)}{\vartheta(v^{-1}a^{-1})\vartheta(vz)}
	\end{array}\right),\\
	\mr{Stab}^{ell}_{-\mf{X}}([1,1])&=\left(\begin{array}{c}
		0\\
		\vartheta(a^{2})\vartheta(v^{-2}z^{-2})
	\end{array}\right).
\end{align*}
\end{cor}

\begin{proof}
By switching the role of $x$ and $y$, we obtain an automorphism $\varpi:X\rightarrow X$ which satisfies $\varpi([2])=[1,1]$, $\varpi([1,1])=[2]$, $\varpi^{\ast}(a)=a^{-1}$, and $\varpi^{\ast}(\mca{O}(1))\cong\mca{O}(1)$. This implies
\begin{align}\label{ell_Unstab}
\mr{Stab}^{ell}_{-\mf{X}}(p_1)|_{p_2}=\left.\left(\mr{Stab}^{ell}_{\mf{X}}\left(\varpi(p_1)\right)|_{\varpi(p_2)}\right)\right|_{a\mapsto a^{-1}},
\end{align}
and hence the result follows from Proposition~\ref{Hilb_ell_stab}.
\end{proof}

\subsection{$K$-theory limits}

Next we describe the $K$-theory limits of elliptic stable basis for any slope $s\in\bb{X}_{\ast}(K)\otimes_{\bb{Z}}\bb{R}\cong \bb{R}$. In this case, $s$ is generic if and only if $2s\notin\bb{Z}$.

\begin{prop}
\begin{enumerate}
	\item If $m<s<m+\frac{1}{2}$ for $m\in\bb{Z}$, then 
		\begin{align*}
		\sqrt{\mca{L}(\kappa)}\otimes\mr{Stab}^K_{\mf{X},s}([2])&=\left(\begin{array}{c}
		v^{2m}(a-a^{-1})\\
		0\end{array}\right),\\
		\sqrt{\mca{L}(\kappa)}\otimes\mr{Stab}^K_{\mf{X},s}([1,1])&=\left(\begin{array}{c}
		v^{2m}a^{-2m}(v-v^{-1})\\
		v^{2m}(va-v^{-1}a^{-1})\end{array}\right).
		\end{align*}
	\item If $m+\frac{1}{2}<s<m+1$ for $m\in\bb{Z}$, then
		\begin{align*}
		\sqrt{\mca{L}(\kappa)}\otimes\mr{Stab}^K_{\mf{X},s}([2])&=\left(\begin{array}{c}
		v^{2m+1}(a-a^{-1})\\
		0\end{array}\right),\\
		\sqrt{\mca{L}(\kappa)}\otimes\mr{Stab}^K_{\mf{X},s}([1,1])&=\left(\begin{array}{c}
		v^{2m+1}a^{-2m-2}(v-v^{-1})\\
		v^{2m+1}(va-v^{-1}a^{-1})\end{array}\right).
		\end{align*}

	\item If $s=m\in\bb{Z}$, then
		\begin{align*}
		\sqrt{\mca{L}(\kappa)}\otimes\mr{Stab}^K_{\mf{X},s}([2])&=\left(\begin{array}{c}
		v^{2m}(a-a^{-1})\frac{1-v^{-2}z^{-2}}{1-v^{-1}z^{-2}}\\
		0\end{array}\right),\\
		\sqrt{\mca{L}(\kappa)}\otimes\mr{Stab}^K_{\mf{X},s}([1,1])&=\left(\begin{array}{c}
		v^{2m}a^{-2m}(v-v^{-1})\frac{(1+az^{-1})(1-a^{-1}z^{-1})}{1-vz^{-2}}\\
		v^{2m}(va-v^{-1}a^{-1})\frac{1-z^{-2}}{1-vz^{-2}}\end{array}\right),\end{align*}
		\item If $s=m+\frac{1}{2}$ for $m\in\bb{Z}$, then
		\begin{align*}
		\sqrt{\mca{L}(\kappa)}\otimes\mr{Stab}^K_{\mf{X},s}([2])&=\left(\begin{array}{c}
		v^{2m+1}(a-a^{-1})\frac{1-v^{-2}z^{-2}}{1-v^{-1}z^{-2}}\\
		0\end{array}\right),\\
		\sqrt{\mca{L}(\kappa)}\otimes\mr{Stab}^K_{\mf{X},s}([1,1])&=\left(\begin{array}{c}
		v^{2m+1}a^{-2m-1}(v-v^{-1})\frac{(a^{-1}+z^{-1})(1-az^{-1})}{1-vz^{-2}}\\
		v^{2m+1}(va-v^{-1}a^{-1})\frac{1-z^{-2}}{1-vz^{-2}}\end{array}\right).\end{align*}
\end{enumerate}
\end{prop}

\begin{proof}
	This follows by straightforward calculations from Proposition~\ref{Hilb_ell_stab} and
\begin{align*}
	\lim_{q\rightarrow0}\frac{\vartheta(xyq^s)}{\vartheta(yq^s)}=\begin{cases}
		x^{-\lfloor s\rfloor-\frac{1}{2}} &\mbox{ if }s\notin\bb{Z},\\
		x^{-s-\frac{1}{2}}\frac{1-xy}{1-y} &\mbox{ if }s\in\bb{Z}.
	\end{cases}
\end{align*}
\end{proof}

By using (\ref{ell_Unstab}) and $vN_{\varpi(p)^!,-}=vN_{p^!_{\mr{flop}},-}$, we also obtain a formula for $\mr{Stab}^K_{-\mf{X},s}(p)$. In a matrix form, it is given by 
\begin{align*}
	\sqrt{\mca{L}(\kappa)}\otimes\mr{Stab}^K_{-\mf{X},s}=\left(\begin{array}{cc}
		0 & 1\\
		1 & 0\end{array}\right)\cdot\left.\left(\sqrt{\mca{L}(\kappa)}\otimes\mr{Stab}^K_{\mf{X},s}\right)\right|_{a\mapsto a^{-1}}\cdot\left(\begin{array}{cc}
		0 & 1\\
		1 & 0\end{array}\right),
\end{align*}
where we set $\mr{Stab}^K_{\mf{X},s}\coloneqq\left(\mr{Stab}^K_{\mf{X},s}([2]),\mr{Stab}^K_{\mf{X},s}([1,1])\right)$.

\subsection{$K$-theoretic canonical bases}

We compute the $K$-theoretic canonical bases for any slopes. For generic slopes, one can apply Kazhdan-Lusztig type algorithm as in \cite[Proposition 3.21]{H1} to compute the dual of $K$-theoretic canonical bases. Once we find the answer on computer, it is not difficult to check it by hand.

\begin{prop}\label{K-th_can_gen}
\begin{enumerate}
	\item If $m<s<m+\frac{1}{2}$ for $m\in\bb{Z}$, then
	\begin{align*}
		\sqrt{\mca{L}(\kappa)}\otimes\mca{E}^K_{\mf{X},s}([2])&=va^{m+\frac{1}{2}}\,\mca{O}\left(m-\frac{1}{2}\right)=\left(\begin{array}{c}
		v^{2m}a\\
		v^{2m}a^{2m}\end{array}\right),\\
		\sqrt{\mca{L}(\kappa)}\otimes\mca{E}^K_{\mf{X},s}([1,1])&=a^{-m+\frac{1}{2}}\,\mca{O}\left(m+\frac{1}{2}\right)=\left(\begin{array}{c}
		v^{2m+1}a^{-2m}\\
		v^{2m+1}a\end{array}\right).
	\end{align*}
	\item If $m+\frac{1}{2}<s<m+1$ for $m\in\bb{Z}$, then
	\begin{align*}
		\sqrt{\mca{L}(\kappa)}\otimes\mca{E}^K_{\mf{X},s}([2])&=a^{m+\frac{3}{2}}\,\mca{O}\left(m+\frac{1}{2}\right)=\left(\begin{array}{c}
		v^{2m+1}a\\
		v^{2m+1}a^{2m+2}\end{array}\right),\\
		\sqrt{\mca{L}(\kappa)}\otimes\mca{E}^K_{\mf{X},s}([1,1])&=v^{-1}a^{-m-\frac{1}{2}}\,\mca{O}\left(m+\frac{3}{2}\right)=\left(\begin{array}{c}
		v^{2m+2}a^{-2m-2}\\
		v^{2m+2}a\end{array}\right).
	\end{align*}
\end{enumerate}
\end{prop}

\begin{proof}
Let us write $\mca{E}^K_{\mf{X},s}\coloneqq\left(\mca{E}^K_{\mf{X},s}([2]),\mca{E}^K_{\mf{X},s}([1,1])\right)$, where $\mca{E}^K_{\mf{X},s}([2])$ and $\mca{E}^K_{\mf{X},s}([1,1])$ are given by the formula in the statement. It is straightforward to check that 
\begin{align*}
	\left(\mca{E}^K_{\mf{X},s}\right)^{-1}\cdot\mr{Stab}^K_{\mf{X},s}=\begin{cases}
		\left(\begin{array}{cc}
		1 & -v^{-1}a^{-2m-1}\\
		-v^{-1}a^{2m-1} & 1\end{array}\right) \mbox{ if }m<s<m+\frac{1}{2},\\
		\left(\begin{array}{cc}
		1 & -v^{-1}a^{-2m-3}\\
		-v^{-1}a^{2m+1} & 1\end{array}\right) \mbox{ if }m+\frac{1}{2}<s<m+1,
	\end{cases}
\end{align*}
and 
\begin{align*}
	\left(\mca{E}^K_{\mf{X},s}\right)^{-1}\cdot\left(-v\cdot\mr{Stab}^K_{-\mf{X},s}\right)=\begin{cases}
		\left(\begin{array}{cc}
		1 & -va^{-2m-1}\\
		-va^{2m-1} & 1\end{array}\right) \mbox{ if }m<s<m+\frac{1}{2},\\
		\left(\begin{array}{cc}
		1 & -va^{-2m-3}\\
		-va^{2m+1} & 1\end{array}\right) \mbox{ if }m+\frac{1}{2}<s<m+1.
	\end{cases}
\end{align*}
This proves that $\mca{E}^K_{\mf{X},s}$ satisfies the properties of $K$-theoretic canonical basis in Definition~\ref{K-th_can}. 
\end{proof}

For $K$-theoretic canonical basis at non-generic slope $s$, one can also apply Kazhdan-Lusztig type algorithm with respect to the bar involution $\beta^K_{\mf{X},s}$ and ``standard'' basis $\{\mca{E}^K_{\mf{X},s_+}(p)\}_{p\in X^H}$. The result is given as follows.

\begin{prop}\label{prop_wall_K-th_can}
\begin{enumerate}
	\item If $s=m\in\bb{Z}$, then
	\begin{align*}
		\sqrt{\mca{L}(\kappa)}\otimes\mca{E}^K_{\mf{X},s}([2])&=va^{m+\frac{1}{2}}\,\mca{O}\left(m-\frac{1}{2}\right)-z^{-1}v^{-1}a^{m+\frac{1}{2}}\,\mca{O}\left(m+\frac{1}{2}\right)\\
		&=\left(\begin{array}{c}
		v^{2m}a-z^{-1}v^{2m}\\
		v^{2m}a^{2m}-z^{-1}v^{2m}a^{2m+1}\end{array}\right),\\
		\sqrt{\mca{L}(\kappa)}\otimes\mca{E}^K_{\mf{X},s}([1,1])&=a^{-m+\frac{1}{2}}\,\mca{O}\left(m+\frac{1}{2}\right)-z^{-1}a^{-m+\frac{1}{2}}\,\mca{O}\left(m-\frac{1}{2}\right)\\
		&=\left(\begin{array}{c}
		v^{2m+1}a^{-2m}-z^{-1}v^{2m-1}a^{-2m+1}\\
		v^{2m+1}a-z^{-1}v^{2m-1}\end{array}\right),
	\end{align*}
	\item If $s=m+\frac{1}{2}$ for $m\in\bb{Z}$, then
	\begin{align*}
		\sqrt{\mca{L}(\kappa)}\otimes\mca{E}^K_{\mf{X},s}([2])&=a^{m+\frac{3}{2}}\,\mca{O}\left(m+\frac{1}{2}\right)=\left(\begin{array}{c}
		v^{2m+1}a\\
		v^{2m+1}a^{2m+2}\end{array}\right),\\
		\sqrt{\mca{L}(\kappa)}\otimes\mca{E}^K_{\mf{X},s}([1,1])&=v^{-1}a^{-m-\frac{1}{2}}\,\mca{O}\left(m+\frac{3}{2}\right)+z^{-2}va^{-m-\frac{1}{2}}\,\mca{O}\left(m-\frac{1}{2}\right)\\
		&=\left(\begin{array}{c}
		v^{2m+2}a^{-2m-2}+z^{-2}v^{2m}a^{-2m}\\
		v^{2m+2}a+z^{-2}v^{2m}a^{-1}\end{array}\right),
	\end{align*}
\end{enumerate}	
\end{prop}

\begin{proof}
As in the proof of Proposition~\ref{K-th_can_gen}, we write $\mca{E}^K_{\mf{X},s}\coloneqq\left(\mca{E}^K_{\mf{X},s}([2]),\mca{E}^K_{\mf{X},s}([1,1])\right)$, where $\mca{E}^K_{\mf{X},s}([2])$ and $\mca{E}^K_{\mf{X},s}([1,1])$ are given by the formula in the statement. By straightforward calculations, we obtain
\begin{align*}
\left(\mca{E}^K_{\mf{X},s}\right)^{-1}\cdot\mr{Stab}^K_{\mf{X},s}&=
		\left(\begin{array}{cc}
		\frac{1-v^{-2}a^{-1}z^{-1}}{1-v^{-1}z^{-2}} & -v^{-1}a^{-2m-1}\frac{1-v^2az^{-1}}{1-vz^{-2}}\\
		-v^{-1}a^{2m-1}\frac{1-az^{-1}}{1-v^{-1}z^{-2}} & \frac{1-a^{-1}z^{-1}}{1-vz^{-2}}\end{array}\right),\\
		\left(\mca{E}^K_{\mf{X},s}\right)^{-1}\cdot\left(-v\cdot\mr{Stab}^K_{-\mf{X},s}\right)&=
		\left(\begin{array}{cc}
		\frac{1-v^{2}a^{-1}z^{-1}}{1-vz^{-2}} & -va^{-2m-1}\frac{1-v^{-2}az^{-1}}{1-v^{-1}z^{-2}}\\
		-va^{2m-1}\frac{1-az^{-1}}{1-vz^{-2}} & \frac{1-a^{-1}z^{-1}}{1-v^{-1}z^{-2}}\end{array}\right)
\end{align*}
when $s=m\in\bb{Z}$ and
\begin{align*}
\left(\mca{E}^K_{\mf{X},s}\right)^{-1}\cdot\mr{Stab}^K_{\mf{X},s}&=
		\left(\begin{array}{cc}
		\frac{1+v^{-2}a^{-2}z^{-2}}{1-v^{-1}z^{-2}} & -v^{-1}a^{-2m-3}\frac{1+v^2a^2z^{-2}}{1-vz^{-2}}\\
		-v^{-1}a^{2m+1}\frac{1}{1-v^{-1}z^{-2}} & \frac{1}{1-vz^{-2}}\end{array}\right),\\
		\left(\mca{E}^K_{\mf{X},s}\right)^{-1}\cdot\left(-v\cdot\mr{Stab}^K_{-\mf{X},s}\right)&=
		\left(\begin{array}{cc}
		\frac{1+v^{2}a^{-2}z^{-2}}{1-vz^{-2}} & -va^{-2m-3}\frac{1+v^{-2}a^2z^{-2}}{1-v^{-1}z^{-2}}\\
		-va^{2m+1}\frac{1}{1-vz^{-2}} & \frac{1}{1-v^{-1}z^{-2}}\end{array}\right)
\end{align*}
when $s=m+\frac{1}{2}$ for $m\in\bb{Z}$. This proves that $\mca{E}^K_{\mf{X},s}([2])$ and $\mca{E}^K_{\mf{X},s}([1,1])$ are $\beta^K_{\mf{X},s}$-invariant. They have the desired form in Conjecture~\ref{conj_wall_can} by Proposition~\ref{K-th_can_gen}. 
\end{proof}

We now describe the index set $\Xi_{\mf{X}}$ in our case.

\begin{cor}
We have 
\begin{align*}
	\widetilde{\bb{B}}^K_{\mf{X}}=\left\{a^m\mca{O}(n),v^{-1}a^m\mca{O}(n),va^m\mca{O}(n)\middle|m,n\in\bb{Z}\right\}
\end{align*}
and the equivalence relation $\sim$ on $\widetilde{\bb{B}}^K_{\mf{X}}$ is given by $v^{\epsilon}a^m\mca{O}(n)\sim v^{\epsilon'}a^m\mca{O}(n')$ and $a^m\mca{O}(n)\sim a^m\mca{O}(n')$ for any $m,n,n'\in\bb{Z}$ and $\epsilon,\epsilon'\in\{\pm1\}$.
\end{cor}

\begin{proof}
By Proposition~\ref{prop_wall_K-th_can}, the equivalence relation is generated by 
\begin{align*}
	va^m\mca{O}(n)\sim v^{-1}a^m\mca{O}(n+1),\,a^m\mca{O}(n)\sim a^m\mca{O}(n+1),\,v^{-1}a^m\mca{O}(n)\sim va^m\mca{O}(n-2)
\end{align*}
for any $m,n\in\bb{Z}$. The result follows from this.
\end{proof}

We identify $\iota_{\mf{X}}:\Xi_{\mf{X}}\xrightarrow{\sim}X^H\cong\{[2],[1,1]\}$ by associating $[a^m\mca{O}(n)]$ to $[1,1]$ and $[v^{\pm1}a^m\mca{O}(n)]$ to $[2]$. We note that the composition $\Xi_{\mf{X}}\xrightarrow{\iota_{\mf{X}}} X^H\cong(X^!)^{H^!}\xrightarrow{\iota_{\mf{X}^!}^{-1}}\Xi_{\mf{X}^!}$ sends $[a^m\mca{O}(n)]$ to $[v^{\pm1}z^m\mca{O}(n)]$ and $[v^{\pm1}a^m\mca{O}(n)]$ to $[z^m\mca{O}(n)]$. The action of $\bb{X}_{\ast}(K)$ on $\Xi_{\mf{X}}$ is trivial in this case.

\section{Elliptic canonical bases for Hilbert scheme of 2-points}

\subsection{Main result}

We now calculate elliptic canonical basis for $X$. We first summarize the expected properties of elliptic canonical basis in our case. By Proposition~\ref{prop_wall_K-th_can}, $\mca{E}^{ell}_{\mf{X}}([2])$ and $\mca{E}^{ell}_{\mf{X}}([1,1])$ should contain the following terms up to character twists for some $r_1,r_2,r_3:\bb{Z}\rightarrow\bb{R}$:
\begin{align*}
	\mca{E}^{ell}_{\mf{X}}([2])&=\sum_{m\in\bb{Z}}(-1)^mq^{r_1(m)}z^{3m+\frac{1}{2}}v\,\mca{O}\left(m-\frac{1}{2}\right)\\
	&\hspace{1em}-\sum_{m\in\bb{Z}}(-1)^mq^{r_2(m)}z^{3m-\frac{1}{2}}v^{-1}\,\mca{O}\left(m+\frac{1}{2}\right)+\cdots,\\
	\mca{E}^{ell}_{\mf{X}}([1,1])&=\sum_{m\in\bb{Z}}(-1)^mq^{r_3(m)}z^{m+\frac{1}{2}}\mca{O}\left(m+\frac{1}{2}\right)+\cdots.
\end{align*}
By looking at the leading term at $z=q^{-s}$ for $s=m,m+\frac{1}{2}$, we obtain 
\begin{align*}
	r_1(m)-\left(3m+\frac{1}{2}\right)m&=r_2(m)-\left(3m-\frac{1}{2}\right)m,\\
	r_1(m)-\left(3m+\frac{1}{2}\right)\left(m+\frac{1}{2}\right)&=r_2(m+1)-\left(3m+\frac{5}{2}\right)\left(m+\frac{1}{2}\right),\\
	r_3(m)-\left(m+\frac{1}{2}\right)m&=r_3(m-1)-\left(m-\frac{1}{2}\right)m.
\end{align*}
Hence we obtain $r_1(m)=\frac{3}{2}\left(m+\frac{1}{6}\right)^2$, $r_2(m)=\frac{3}{2}\left(m-\frac{1}{6}\right)^2$, and $r_3(m)=\frac{1}{2}\left(m+\frac{1}{2}\right)^2$ up to constant term shifts. In this normalization, the $q$-difference equations in Property~D should be
\begin{align}
	\delta_z\left(\mca{E}^{ell}_{\mf{X}}([2])\right)&=-q^{-\frac{3}{2}}z^{-3}\mca{O}(-1)\otimes\mca{E}^{ell}_{\mf{X}}([2]),\label{q-diff1}\\
	\delta_z\left(\mca{E}^{ell}_{\mf{X}}([1,1])\right)&=-q^{-\frac{1}{2}}z^{-1}\mca{O}(-1)\otimes\mca{E}^{ell}_{\mf{X}}([1,1]),\label{q-diff2}
\end{align}
by looking at the $q$-difference equations satisfied by the leading terms of $\mca{E}^{ell}_{\mf{X}}(\mu)$. As for Property~A, we do not determine the multi-valuedness on $q$ in this paper and as a vector-valued function on $H\times K\times \bb{S}$, we assume
\begin{align}\label{multivaluedness}
	z^{-\frac{1}{2}}\mca{O}\left(-\frac{1}{2}\right)\otimes\mca{E}^{ell}_{\mf{X}}(\mu)\mbox{ is single valued and holomorphic}.
\end{align} 
For $K$-theory limits, we fix the normalization by 
\begin{align}
	\lim_{q\rightarrow0}q^{-r_{[2]}(s)}\delta_z^{-s}\left(\mca{E}^{ell}_{\mf{X}}([2])\right)&=vz^{\frac{1}{2}}\,\mca{O}\left(-\frac{1}{2}\right),\label{K-limit1}\\
	\lim_{q\rightarrow0}q^{-r_{[1,1]}(s)}\delta_z^{-s}\left(\mca{E}^{ell}_{\mf{X}}([1,1])\right)&=z^{\frac{1}{2}}\,\mca{O}\left(\frac{1}{2}\right),\label{K-limit2}
\end{align}
for some $r_{[2]}(s),r_{[1,1]}(s)\in\bb{R}$ for any $0<s<\frac{1}{2}$.

The elliptic canonical bases for $\mf{X}^!$ is obtained by switching $a$ and $z$, i.e, we have
\begin{align*}
	\mca{E}^{ell}_{\mf{X}^!}(\mu)=\left(\begin{array}{c}
	\left.\mca{E}^{ell}_{\mf{X}^!}(\mu)\right|_{[2]^!}\\
	\left.\mca{E}^{ell}_{\mf{X}^!}(\mu)\right|_{[1,1]^!}
	\end{array}\right)=\left.\left(\begin{array}{c}
	\left.\mca{E}^{ell}_{\mf{X}}(\mu)\right|_{[1,1]}\\
	\left.\mca{E}^{ell}_{\mf{X}}(\mu)\right|_{[2]}
	\end{array}\right)\right|_{a\leftrightarrow z}.
\end{align*}
Since we have $\sigma_{\mf{X},\mf{X}^!}(p)=1$ in our case, Property~B becomes 
\begin{align}\label{duality}
	\Upsilon(v;q)\cdot\mr{Stab}^{ell}_{\mf{X}}=\mca{E}^{ell}_{\mf{X}}\cdot^t\!\mca{E}^{ell}_{\mf{X}^!}
\end{align}
for some (possibly multi-valued) holomorphic function $\Upsilon(v)=\Upsilon(v;q)$ on $\bb{S}\times\mf{D}^{\circ}$ and
\begin{align*}
\mr{Stab}^{ell}_{\mf{X}}&=\left(\begin{array}{cc}
	\left.\mr{Stab}^{ell}_{\mf{X}}([2])\right|_{[2]} & \left.\mr{Stab}^{ell}_{\mf{X}}([1,1])\right|_{[2]} \\
	\left.\mr{Stab}^{ell}_{\mf{X}}([2])\right|_{[1,1]} & \left.\mr{Stab}^{ell}_{\mf{X}}([1,1])\right|_{[1,1]}
\end{array}\right),\\
\mca{E}^{ell}_{\mf{X}}&=\left(\begin{array}{cc}
	\left.\mca{E}^{ell}_{\mf{X}}([2])\right|_{[2]} & \left.\mca{E}^{ell}_{\mf{X}}([1,1])\right|_{[2]} \\
	\left.\mca{E}^{ell}_{\mf{X}}([2])\right|_{[1,1]} & \left.\mca{E}^{ell}_{\mf{X}}([1,1])\right|_{[1,1]}
\end{array}\right),\\
\mca{E}^{ell}_{\mf{X}^!}&=\left(\begin{array}{cc}
	\left.\mca{E}^{ell}_{\mf{X}^!}([1,1])\right|_{[1,1]} & \left.\mca{E}^{ell}_{\mf{X}^!}([2])\right|_{[1,1]} \\
	\left.\mca{E}^{ell}_{\mf{X}^!}([1,1])\right|_{[2]} & \left.\mca{E}^{ell}_{\mf{X}^!}([2])\right|_{[2]}
\end{array}\right).
\end{align*}
Our main result classifies the solutions of these equations.

\begin{thm}\label{Main_Thm}
$\mca{E}^{ell}_{\mf{X}}$ satisfies (\ref{q-diff1}), (\ref{q-diff2}), (\ref{multivaluedness}), (\ref{K-limit1}), (\ref{K-limit2}), (\ref{duality}), and Property~A if and only if there exist (possibly multi-valued) holomorphic functions $f_i(v;q)$ on $\bb{S}\times\widetilde{\mf{D}^{\circ}}$, $i=0,1,2$, such that $q^{-c_i}f_i(v;q)$ is single-valued and
\begin{align*}
	f_0(v;q)&=q^{c_0}\cdot\left(1+O(q^{\frac{1}{2}})\right),\\
	f_1(v;q)&=q^{c_1}\cdot\left(1+O(q^{\frac{1}{2}})\right),\\
    f_2(v;q)&=q^{c_2}\cdot O(1),
\end{align*}
for some $c_0,c_1,c_2\in\bb{R}$ satisfying $c_2\geq c_1+\frac{3}{4}$, and we have
\begin{align*}
	\mca{E}^{ell}_{\mf{X}}([2])&=f_1(v;q)\sum_{l,m\in\bb{Z}}(-1)^{m}q^{l^2+\frac{1}{2}\left(m+\frac{1}{2}\right)^2}v^{-2l+2m+1}z^{2l+m+\frac{1}{2}}\mca{O}\left(2l-m-\frac{1}{2}\right)\\
	&\hspace{1em}+f_2(v;q)\sum_{l,m\in\bb{Z}}(-1)^mq^{\left(l+\frac{1}{2}\right)^2+\frac{1}{2}\left(m+\frac{1}{2}\right)^2}v^{-2l+2m}z^{2l+m+\frac{3}{2}}\mca{O}\left(2l-m+\frac{1}{2}\right),\\
	\mca{E}^{ell}_{\mf{X}}([1,1])&=f_0(v;q)\sum_{m\in\bb{Z}}(-1)^mq^{\frac{1}{2}\left(m+\frac{1}{2}\right)^2}z^{m+\frac{1}{2}}\mca{O}\left(m+\frac{1}{2}\right),\\
	\Upsilon(v;q)&=f_0(v;q)\left(f_1(v;q)\sum_{m\in\bb{Z}}q^{m^2}v^{2m}+f_2(v;q)\sum_{m\in\bb{Z}}q^{\left(m+\frac{1}{2}\right)^2}v^{2m+1}\right).
\end{align*}
\end{thm}

The rest of this paper is devoted to a proof of this theorem. In this section, we always assume (\ref{q-diff1}), (\ref{q-diff2}), (\ref{multivaluedness}), (\ref{K-limit1}), (\ref{K-limit2}) and (\ref{duality}).

\subsection{Solving $q$-difference equations}

We first solve the $q$-difference equations on the K\"ahler variable $z$ and also on the equivariant variable $a$. By (\ref{multivaluedness}), we can expand 
\begin{align*}
	\mca{E}^{ell}_{\mf{X}}([2])&=\sum_{\lambda=\pm\frac{1}{6},\frac{1}{2}}\sum_{m\in\bb{Z}}(-1)^mq^{\frac{3}{2}\left(m+\lambda\right)^2}z^{3(m+\lambda)}\mca{O}(m+\lambda)\otimes\mca{F}_{m,\lambda}(v,a;q),\\
	\mca{E}^{ell}_{\mf{X}}([1,1])&=\sum_{m\in\bb{Z}}(-1)^mq^{\frac{1}{2}\left(m+\frac{1}{2}\right)^2}z^{m+\frac{1}{2}}\mca{O}\left(m+\frac{1}{2}\right)\otimes\mca{F}_m(v,a;q),
\end{align*}
for some $\mca{F}_{m,\lambda}(v,a;q),\mca{F}_m(v,a;q)\in\mb{E}(\mf{X})_{\mr{loc}}$ which do not depend on $z$. By (\ref{q-diff1}), we obtain $\mca{F}_{m+1,\lambda}(v,a;q)=\mca{F}_{m,\lambda}(v,a;q)\eqqcolon\mca{F}_{\lambda}(v,a;q)$ and $\mca{F}_{m+1}(v,a;q)=\mca{F}_{m}(v,a;q)\eqqcolon\mca{F}(v,a;q)$ for any $m\in\bb{Z}$. By setting 
\begin{align*}
	\mca{E}^{[2]}_{\lambda}&\coloneqq\sum_{m\in\bb{Z}}(-1)^mq^{\frac{3}{2}\left(m+\lambda\right)^2}z^{3(m+\lambda)}\mca{O}(m+\lambda),\\
	\mca{E}^{[1,1]}&\coloneqq\sum_{m\in\bb{Z}}(-1)^mq^{\frac{1}{2}\left(m+\frac{1}{2}\right)^2}z^{m+\frac{1}{2}}\mca{O}\left(m+\frac{1}{2}\right),
\end{align*}
for $\lambda\in\{\pm\frac{1}{6},\frac{1}{2}\}$, we obtain
\begin{align}
	\mca{E}^{ell}_{\mf{X}}([2])&=\sum_{\lambda=\pm\frac{1}{6},\frac{1}{2}}\mca{F}_{\lambda}(v,a;q)\otimes\mca{E}^{[2]}_{\lambda},\label{eq_z-q-diff}\\
	\mca{E}^{ell}_{\mf{X}}([1,1])&=\mca{F}(v,a;q)\otimes\mca{E}^{[1,1]}.\notag
\end{align}

\begin{lem}\label{lem_equiv_q-diff}
$\mca{E}_{\mf{X}}^{ell}$	 satisfies the following $q$-difference equation:
\begin{align*}
	\delta_{a}\left(\mca{E}^{ell}_{\mf{X}}\right)=\left(\begin{array}{cc}
		v^2z & 0\\
		0 & v^{-2}z^{-1}
	\end{array}\right)\mca{E}^{ell}_{\mf{X}}\left(\begin{array}{cc}
		-q^{-\frac{3}{2}}a^{-3} & 0\\
		0 & -q^{-\frac{1}{2}}a^{-1}
	\end{array}\right).
\end{align*}
\end{lem}

\begin{proof}
	By definition or explicit formula in Proposition~\ref{Hilb_ell_stab}, we have
	\begin{align*}
		\delta_a\left(\mr{Stab}^{ell}_{\mf{X}}\right)=q^{-2}a^{-4}\left(\begin{array}{cc}
		v^2z & 0\\
		0 & v^{-2}z^{-1}
	\end{array}\right)\mr{Stab}^{ell}_{\mf{X}}\left(\begin{array}{cc}
		v^{-2}z^{-1} & 0\\
		0 & v^{-2}z
	\end{array}\right).
	\end{align*} 
By applying (\ref{q-diff1}) and (\ref{q-diff2}) to $\mf{X}^!$, we obtain
\begin{align*}
	\delta_{a}\left(\mca{E}^{ell}_{\mf{X}^!}\right)=\left(\begin{array}{cc}
		v^{-2}z^{-1} & 0\\
		0 & v^{-2}z
	\end{array}\right)\mca{E}^{ell}_{\mf{X}^!}\left(\begin{array}{cc}
		-q^{-\frac{1}{2}}a^{-1} & 0\\
		0 & -q^{-\frac{3}{2}}a^{-3}
	\end{array}\right).
\end{align*}
Hence by (\ref{duality}), we obtain
\begin{align*}
\delta_a\left(\mca{E}^{ell}_{\mf{X}}\right)&=\delta_a\left(\Upsilon(v;q)\cdot\mr{Stab}^{ell}_{\mf{X}}\cdot\left(^t\mca{E}^{ell}_{\mf{X}^!}\right)^{-1}\right)\\
&=q^{-2}a^{-4}\left(\begin{array}{cc}
		v^2z & 0\\
		0 & v^{-2}z^{-1}
	\end{array}\right)\mca{E}^{ell}_{\mf{X}}\left(\begin{array}{cc}
		-q^{\frac{1}{2}}a & 0\\
		0 & -q^{\frac{3}{2}}a^{3}
	\end{array}\right).
\end{align*}
\end{proof}

By Lemma~\ref{lem_equiv_q-diff} and 
\begin{align*}
	\delta_a\left(\left.\mca{E}^{[2]}_{\lambda}\right|_p\right)=q^{-\frac{1}{6}}z^{\epsilon_p}v^{\frac{2}{3}\epsilon_p}a^{-\frac{1}{3}}\left.\mca{E}^{[2]}_{\lambda-\frac{1}{3}\epsilon_p}\right|_{p},\\
	\delta_a\left(\left.\mca{E}^{[1,1]}\right|_p\right)=-q^{-\frac{1}{2}}v^{2\epsilon_p}a^{-1}z^{\epsilon_p}\left.\mca{E}^{[1,1]}\right|_p,
\end{align*}
where we set $\epsilon_{[2]}=1$, $\epsilon_{[1,1]}=-1$, and $\mca{E}^{[2]}_{\lambda\pm1}=-\mca{E}^{[2]}_{\lambda}$, we obtain 
\begin{align}
	\delta_a\left(\left.\mca{F}_{\lambda}(v,a;q)\right|_{p}\right)&=-q^{-\frac{4}{3}}v^{\frac{4}{3}\epsilon_p}a^{-\frac{8}{3}}\left.\mca{F}_{\lambda-\frac{1}{3}\epsilon_p}(v,a;q)\right|_{p},\label{eq_a-q-diff}\\
	\delta_a\left(\mca{F}(v,a;q)\right)&=\mca{F}(v,a;q),\notag
\end{align}
where $\mca{F}_{\lambda\pm1}(v,a;q)=-\mca{F}_{\lambda}(v,a;q)$. In particular, we have
\begin{align*}
	\delta_a^3\left(\left.\mca{F}_{\lambda}(v,a;q)\right|_{p}\right)=q^{-12}v^{4\epsilon_p}a^{-8}\left.\mca{F}_{\lambda}(v,a;q)\right|_{p}.
\end{align*}
Since $\mca{F}_{\frac{1}{2}}(v,a;q)$ and $\mca{F}(v,a;q)$ are single valued and holomorphic by (\ref{multivaluedness}), these $q$-difference equations imply that $\mca{F}(v,a;q)$ does not depend on $a$ and we can expand 
\begin{align*}
	\left.\mca{F}_{\frac{1}{2}}(v,a;q)\right|_p=\sum_{\lambda\in\frac{1}{8}\bb{Z}/\bb{Z}}h_{\lambda}^p(v;q)g^p_{\lambda}(v,a;q)
\end{align*}
for some $h^p_{\lambda}(v;q)\in F$ which do not depend on $a$ and $z$, where we set
\begin{align*}
g^p_{\lambda}(v,a;q)\coloneqq\sum_{l\in\bb{Z}}q^{12(l+\lambda)^2}v^{4(l+\lambda)}a^{-8(l+\lambda)\epsilon_p}
\end{align*}
for any $\lambda\in\bb{Q}$. Since we have $\delta_a\left(g^p_{\lambda}(v,a;q)\right)=q^{-\frac{4}{3}}v^{\frac{4}{3}\epsilon_p}a^{-\frac{8}{3}}g^p_{\lambda-\frac{1}{3}\epsilon_p}(v,a;q)$, we also obtain 
\begin{align*}
	\left.\mca{F}_{\frac{1}{2}-\frac{1}{3}\epsilon_p}(v,a;q)\right|_p&=-\sum_{\lambda\in\frac{1}{8}\bb{Z}/\bb{Z}}h^p_{\lambda}(v;q)g^p_{\lambda-\frac{1}{3}\epsilon_p}(v,a;q),\\
	\left.\mca{F}_{\frac{1}{2}-\frac{2}{3}\epsilon_p}(v,a;q)\right|_p&=\sum_{\lambda\in\frac{1}{8}\bb{Z}/\bb{Z}}h^p_{\lambda}(v;q)g^p_{\lambda-\frac{2}{3}\epsilon_p}(v,a;q),
\end{align*}
by (\ref{eq_a-q-diff}). By substituting them to (\ref{eq_z-q-diff}), we obtain
\begin{align*}
	\left.\mca{E}^{ell}_{\mf{X}}([2])\right|_p&=\sum_{\lambda\in\frac{1}{8}\bb{Z}/\bb{Z}}h^p_{\lambda}(v;q)\sum_{\mu\in\{0,\frac{1}{3},\frac{2}{3}\}}(-1)^{3\mu}g^p_{\lambda-\mu\epsilon_p}(v,a;q)\left.\mca{E}^{[2]}_{\frac{1}{2}-\mu\epsilon_p}\right|_p\\
	&=\sum_{\lambda,\mu}h^p_{\lambda}(v;q)\sum_{l,m\in\bb{Z}}(-1)^{m+3\mu}q^{12\left(l-\mu\epsilon_p+\lambda\right)^2+\frac{3}{2}\left(m-\mu\epsilon_p+\frac{1}{2}\right)^2}\\
	&\hspace{1em}\times z^{3\left(m-\mu\epsilon_p+\frac{1}{2}\right)}v^{4\left(l-\mu\epsilon_p+\lambda\right)+2\left(m-\mu\epsilon_p+\frac{1}{2}\right)}a^{-8\left(l-\mu\epsilon_p+\lambda\right)\epsilon_p-\left(m-\mu\epsilon_p+\frac{1}{2}\right)\epsilon_p}.
\end{align*}
By identifying $\frac{1}{8}\bb{Z}/\bb{Z}\cong\bb{Z}/8\bb{Z}$ and setting $L=8\left(l-\mu\epsilon_p\right)+(m-\mu\epsilon_p)+8\lambda$ and $M=3(m-\mu\epsilon_p)+1$, a straightforward calculation gives the following.
\begin{lem}\label{lem_ell2}
There exist $h^p_{\lambda}(v;q)$ for $\lambda\in\bb{Z}/8\bb{Z}$ such that we have
\begin{align*}
\left.\mca{E}^{ell}_{\mf{X}}([2])\right|_p=\sum_{L,M\in\bb{Z}}(-1)^{M+1}q^{\frac{1}{16}(L+M+1)^2+\frac{1}{8}(L-M)^2}v^{\frac{L+M+1}{2}}h^p_{\lambda}(v;q)z^{M+\frac{1}{2}}a^{-\left(L+\frac{1}{2}\right)\epsilon_p},
\end{align*}
where $\lambda=L-3M+3\mod8$.
\end{lem}

\subsection{Solving bilinear equations}

In this section, we solve some of the bilinear equations (\ref{duality}). We first consider the $([1,1],[2])$-component of the equation, i.e., 
\begin{align*}
	\left.\mca{E}^{ell}_{\mf{X}}([2])\right|_{[1,1]}\cdot\left.\mca{E}^{ell}_{\mf{X}^!}([1,1])\right|_{[1,1]}+\left.\mca{E}^{ell}_{\mf{X}}([1,1])\right|_{[1,1]}\cdot\left.\mca{E}^{ell}_{\mf{X}^!}([2])\right|_{[1,1]}=0.
\end{align*}
Since we have
\begin{align*}
\left.\mca{E}^{ell}_{\mf{X}}([1,1])\right|_{[1,1]}=h^{[1,1]}(v;q)\sum_{m\in\bb{Z}}(-1)^mq^{\frac{1}{2}\left(m+\frac{1}{2}\right)^2}v^{2m+1}z^{m+\frac{1}{2}}a^{m+\frac{1}{2}}
\end{align*}
where $h^p(v)=h^p(v;q)\coloneqq\mca{F}(v,a;q)|_{p}$, we obtain $\left.\mca{E}^{ell}_{\mf{X}^!}([1,1])\right|_{[1,1]}=\left.\mca{E}^{ell}_{\mf{X}}([1,1])\right|_{[1,1]}$. Since $h^{[1,1]}(v;q)\neq0$ by (\ref{K-limit2}), this implies $\left.\mca{E}^{ell}_{\mf{X}}([2])\right|_{[1,1]}=-\left.\mca{E}^{ell}_{\mf{X}^!}([2])\right|_{[1,1]}$. By Lemma~\ref{lem_ell2}, we obtain 
\begin{align*}
	(-1)^Mh^{[1,1]}_{L-3M+3}(v;q)=-(-1)^{L}h^{[1,1]}_{M-3L+3}(v;q).
\end{align*}
It is straightforward to check that this is equivalent to $h^{[1,1]}_{0}(v;q)=h^{[1,1]}_{4}(v;q)$, $h^{[1,1]}_{2}(v;q)=h^{[1,1]}_{6}(v;q)$, and $h^{[1,1]}_{i}(v;q)=0$ for odd $i$.

We next look at the $([2],[2])$-component of (\ref{duality}), i.e., 
\begin{align*}
	\left.\mca{E}^{ell}_{\mf{X}}([2])\right|_{[2]}\cdot\left.\mca{E}^{ell}_{\mf{X}^!}([1,1])\right|_{[1,1]}+\left.\mca{E}^{ell}_{\mf{X}}([1,1])\right|_{[2]}\cdot\left.\mca{E}^{ell}_{\mf{X}^!}([2])\right|_{[1,1]}=\Upsilon(v)\vartheta(a^{-2})\vartheta(v^{-2}z^{-2}).
\end{align*}
By Lemma~\ref{lem_ell2}, we have
\begin{align*}
\left.\mca{E}^{ell}_{\mf{X}}([2])\right|_{[2]}\cdot\left.\mca{E}^{ell}_{\mf{X}^!}([1,1])\right|_{[1,1]}&=\sum_{L,M,m\in\bb{Z}}(-1)^{M+m+1}q^{\frac{1}{16}(L+M+1)^2+\frac{1}{8}(L-M)^2+\frac{1}{2}\left(m+\frac{1}{2}\right)^2}\\
&\hspace{1em}\times v^{\frac{L+M+1}{2}+2m+1}z^{M+m+1}a^{-L+m}h^{[2]}_{L-3M+3}(v)h^{[1,1]}(v)\\
&=\sum_{A,B,m\in\bb{Z}}(-1)^Aq^{\frac{A^2+B^2}{8}+\left(m-\frac{A+B-2}{4}\right)^2}v^{\frac{A-B}{2}+2m+1}\\
&\hspace{1em}\times z^{A}a^{B}h^{[2]}_{-3A-B+4m+6}(v)h^{[1,1]}(v),
\end{align*}
where we set $A=M+m+1$ and $B=-L+m$ in the last equality, and
\begin{align*}
\left.\mca{E}^{ell}_{\mf{X}}([1,1])\right|_{[2]}\cdot\left.\mca{E}^{ell}_{\mf{X}^!}([2])\right|_{[1,1]}&=\sum_{L,M,m\in\bb{Z}}(-1)^{M+m+1}q^{\frac{1}{16}(L+M+1)^2+\frac{1}{8}(L-M)^2+\frac{1}{2}\left(m+\frac{1}{2}\right)^2}\\
&\hspace{1em}\times v^{\frac{L+M+1}{2}+2m+1}z^{L+m+1}a^{M-m}h^{[1,1]}_{L-3M+3}(v)h^{[2]}(v)\\
&=\sum_{A,B,m\in\bb{Z}}(-1)^{B+1}q^{\frac{A^2+B^2}{8}+\left(m-\frac{A-B-2}{4}\right)^2}v^{\frac{A+B}{2}+2m+1}\\
&\hspace{1em}\times z^{A}a^{B}h^{[1,1]}_{A-3B-4m+2}(v)h^{[2]}(v),
\end{align*}
where we set $A=L+m+1$ and $B=M-m$ in the last equality. On the other hand, by the Jacobi triple product identity
\begin{align*}
	\vartheta(x)\cdot q^{\frac{1}{8}}(q;q)_{\infty}=\sum_{m\in\bb{Z}}(-1)^mq^{\frac{1}{2}\left(m+\frac{1}{2}\right)^2}x^{m+\frac{1}{2}},
\end{align*}
where $(q;q)_{\infty}\coloneqq\prod_{m\geq1}(1-q^m)$, we have
\begin{align*}
	\Upsilon(v)\vartheta(a^{-2})\vartheta(v^{-2}z^{-2})=\Upsilon'(v)\sum_{l,m\in\bb{Z}}(-1)^{l+m}q^{\frac{1}{2}\left(l+\frac{1}{2}\right)^2+\frac{1}{2}\left(m+\frac{1}{2}\right)^2}v^{2m+1}z^{2m+1}a^{2l+1}.
\end{align*}
Here, we set $\Upsilon'(v)\coloneqq \Upsilon(v)\cdot q^{-\frac{1}{4}}(q;q)_{\infty}^{-2}$. By comparing the coefficients of $z^A a^B$, we obtain
\begin{align*}
	&\sum_{m\in\bb{Z}}(-1)^Aq^{\left(m-\frac{A+B-2}{4}\right)^2}v^{\frac{A-B}{2}+2m+1}h^{[2]}_{-3A-B+4m+6}(v)h^{[1,1]}(v)\\
	&+\sum_{m\in\bb{Z}}(-1)^{B+1}q^{\left(m-\frac{A-B-2}{4}\right)^2}v^{\frac{A+B}{2}+2m+1}h^{[1,1]}_{A-3B-4m-2}(v)h^{[2]}(v)\\
	&=\begin{cases}
		0 &\mbox{ if }A\mbox{ or }B\mbox{ is even},\\
		\Upsilon'(v)\cdot(-1)^{\frac{A+B}{2}+1}v^{A} &\mbox{ if }A\mbox{ and }B\mbox{ are odd}.
	\end{cases}
\end{align*}
If $A\not\equiv B\,\mr{mod}\,2$, then by $h^{[1,1]}_{i}(v)=0$ for odd $i$, we have 
\begin{align*}
	h^{[2]}_{-3A-B+6}(v)\sum_{m\in2\bb{Z}}q^{\left(m-\frac{A+B-2}{4}\right)^2}v^{2m}+h^{[2]}_{-3A-B+2}(v)\sum_{m\in2\bb{Z}+1}q^{\left(m-\frac{A+B-2}{4}\right)^2}v^{2m}=0.
\end{align*}
By shifting $A$ to $A-2$ and $B$ to $B+2$, we obtain
\begin{align*}
	h^{[2]}_{-3A-B+6}(v)\sum_{m\in2\bb{Z}+1}q^{\left(m-\frac{A+B-2}{4}\right)^2}v^{2m}+h^{[2]}_{-3A-B+2}(v)\sum_{m\in2\bb{Z}}q^{\left(m-\frac{A+B-2}{4}\right)^2}v^{2m}=0.
\end{align*}
Since the matrix
\begin{align*}
	\left(\begin{array}{cc}
	\sum_{m\in2\bb{Z}}q^{\left(m-\frac{A+B-2}{4}\right)^2}v^{2m} & \sum_{m\in2\bb{Z}+1}q^{\left(m-\frac{A+B-2}{4}\right)^2}v^{2m}\\
	\sum_{m\in2\bb{Z}+1}q^{\left(m-\frac{A+B-2}{4}\right)^2}v^{2m} & \sum_{m\in2\bb{Z}}q^{\left(m-\frac{A+B-2}{4}\right)^2}v^{2m}
\end{array}
\right)
\end{align*}
is invertible, we obtain $h^{[2]}_{i}(v)=0$ for odd $i$.

If both $A$ and $B$ are even, then $h^{[1,1]}_{i+4}(v)=h^{[1,1]}_{i}(v)$ implies
\begin{align*}
	&h^{[1,1]}(v)h^{[2]}_{-3A-B+6}(v)\sum_{m\in2\bb{Z}}q^{\left(m-\frac{A+B-2}{4}\right)^2}v^{2m}\\
	&\hspace{1em}+h^{[1,1]}(v)h^{[2]}_{-3A-B+2}(v)\sum_{m\in2\bb{Z}+1}q^{\left(m-\frac{A+B-2}{4}\right)^2}v^{2m}\\
	&=h^{[2]}(v)h^{[1,1]}_{A-3B-2}(v)\sum_{m\in\bb{Z}}q^{\left(m-\frac{A-B-2}{4}\right)^2}v^{B+2m}\\
&=h^{[2]}(v)h^{[1,1]}_{A-3B-2}(v)\sum_{m\in\bb{Z}}q^{\left(m-\frac{A+B-2}{4}\right)^2}v^{2m},
\end{align*}
where we shift $m$ by $m-\frac{B}{2}$ in the last equality. By also considering the shift $A\mapsto A-2$ and $B\mapsto B+2$ as above, we obtain 
\begin{align}\label{Eqn_h}
	h^{[1,1]}(v)h^{[2]}_{i}(v)=h^{[2]}(v)h^{[1,1]}_{i}(v)
\end{align}
for even $i$. In particular, we have $h^{[2]}_{i+4}(v)=h^{[2]}_{i}(v)$.

If both $A$ and $B$ are odd, then we obtain
\begin{align*}
	\Upsilon'(v)&=(-1)^{\frac{A+B}{2}}h^{[2]}_{A-B+2}(v)h^{[1,1]}(v)\sum_{m\in\bb{Z}}q^{\left(m-\frac{A+B-2}{4}\right)^2}v^{2m+1-\frac{A+B}{2}}\\
	&\hspace{1em}-(-1)^{\frac{A+B}{2}}h^{[1,1]}_{A+B+2}(v)h^{[2]}(v)\sum_{m\in\bb{Z}}q^{\left(m-\frac{A-B-2}{4}\right)^2}v^{2m+1-\frac{A-B}{2}}
\end{align*}
which is equivalent to
\begin{align}\label{eq_Upsilon}
	\Upsilon'(v)=h^{[1,1]}(v)\left(h^{[2]}_{0}(v)\sum_{m\in\bb{Z}}q^{\left(m+\frac{1}{2}\right)^2}v^{2m+1}-h^{[2]}_{2}(v)\sum_{m\in\bb{Z}}q^{m^2}v^{2m}\right)
\end{align}
by (\ref{Eqn_h}). In summary, we obtain the following description of $\mca{E}^{ell}_{\mf{X}}([2])$.
 
\begin{lem}\label{lem_ell2_final}
We have
\begin{align*}
\left.\mca{E}^{ell}_{\mf{X}}([2])\right|_p&=h^p_{0}(v)\sum_{l,m\in\bb{Z}}(-1)^mq^{\left(l+\frac{1}{2}\right)^2+\frac{1}{2}\left(m+\frac{1}{2}\right)^2}v^{2l+1}z^{2l+m+\frac{3}{2}}a^{-\left(2l-m+\frac{1}{2}\right)\epsilon_p}\\
&\hspace{1em}-h^p_{2}(v)\sum_{l,m\in\bb{Z}}(-1)^mq^{l^2+\frac{1}{2}\left(m+\frac{1}{2}\right)^2}v^{2l}z^{2l+m+\frac{1}{2}}a^{-\left(2l-m-\frac{1}{2}\right)\epsilon_p}.
\end{align*} 
\end{lem}

\begin{proof}
	By Lemma~\ref{lem_ell2} and $h^p_{i+4}(v)=h^p_{i}(v)$, $h^p_{\mr{odd}}(v)=0$, we have
\begin{align*}
\left.\mca{E}^{ell}_{\mf{X}}([2])\right|_p&=h^p_0(v)\hspace{-1em}\sum_{\substack{L,M\in\bb{Z}\\L+M\equiv 1\,\mr{mod}\,4}}(-1)^{M+1}q^{\frac{1}{16}(L+M+1)^2+\frac{1}{8}(L-M)^2}v^{\frac{L+M+1}{2}}z^{M+\frac{1}{2}}a^{-\left(L+\frac{1}{2}\right)\epsilon_p}\\
&\hspace{0em}+h^p_2(v)\hspace{-1em}\sum_{\substack{L,M\in\bb{Z}\\L+M\equiv 3\,\mr{mod}\,4}}(-1)^{M+1}q^{\frac{1}{16}(L+M+1)^2+\frac{1}{8}(L-M)^2}v^{\frac{L+M+1}{2}}z^{M+\frac{1}{2}}a^{-\left(L+\frac{1}{2}\right)\epsilon_p}.
\end{align*}
The statement follows by changing variables by $l=\frac{L+M-1}{4}$ and $m=\frac{M-L-1}{2}$ in the first sum, and by $l=\frac{L+M+1}{4}$ and $m=\frac{M-L-1}{2}$ in the second sum.
\end{proof}

\subsection{Theta function identities}

The $([1,1],[1,1])$-component of (\ref{duality}) is obtained by exchanging $a$ and $z$ in the $([2],[2])$-component of (\ref{duality}), hence it is enough to consider the $([2],[1,1])$-component, i.e., 
\begin{align*}
\left.\mca{E}^{ell}_{\mf{X}}([2])\right|_{[2]}\cdot\left.\mca{E}^{ell}_{\mf{X}^!}([1,1])\right|_{[2]}+\left.\mca{E}^{ell}_{\mf{X}}([1,1])\right|_{[2]}\cdot\left.\mca{E}^{ell}_{\mf{X}^!}([2])\right|_{[2]}=\Upsilon(v)\cdot\left.\mr{Stab}^{ell}_{\mf{X}}([1,1])\right|_{[2]}.
\end{align*} 
By using the Jacobi triple product identity, we calculate
\begin{align*}
&q^{-\frac{1}{4}}(q;q)_{\infty}^{-2}\left.\mca{E}^{ell}_{\mf{X}}([2])\right|_{[2]}\cdot\left.\mca{E}^{ell}_{\mf{X}^!}([1,1])\right|_{[2]}=h^{[2]}(v)\vartheta(za)\vartheta(v^2z^{-1}a)\\
&\hspace{1em}\times \left(h^{[2]}_{0}(v)\sum_{l\in\bb{Z}}q^{\left(l+\frac{1}{2}\right)^2}v^{2l+1}z^{2l+1}a^{-2l-1}-h^{[2]}_{2}(v)\sum_{l\in\bb{Z}}q^{l^2}v^{2l}z^{2l}a^{-2l}\right)\\
&=h^{[2]}(v)\vartheta(za)\vartheta(v^2z^{-1}a)\left(h^{[2]}_0(v)\vartheta_1(vza^{-1})-h^{[2]}_2(v)\vartheta_0(vza^{-1})\right),
\end{align*}
 where we set
 \begin{align}\label{Eqn_theta}
 	\vartheta_0(x)\coloneqq\sum_{l\in\bb{Z}}q^{l^2}x^{2l},\hspace{1em}\vartheta_1(x)\coloneqq\sum_{l\in\bb{Z}}q^{\left(l+\frac{1}{2}\right)^2}x^{2l+1}.
\end{align}
By using (\ref{eq_Upsilon}) and Proposition~\ref{Hilb_ell_stab}, we obtain 
\begin{align*}
	&h^{[2]}(v)h^{[2]}_0(v)\left(\vartheta(za)\vartheta(v^2z^{-1}a)\vartheta_1(vza^{-1})+\vartheta(za)\vartheta(v^2za^{-1})\vartheta_1(vz^{-1}a)\right)\\
	&\hspace{0em}-h^{[2]}(v)h^{[2]}_2(v)\left(\vartheta(za)\vartheta(v^2z^{-1}a)\vartheta_0(vza^{-1})+\vartheta(za)\vartheta(v^2za^{-1})\vartheta_0(vz^{-1}a)\right)\\
&=h^{[1,1]}(v)h^{[2]}_0(v)\vartheta_1(v)\vartheta(v^{-2})\frac{\vartheta(a^{-2})\vartheta(vz^2a^{-1})\vartheta(v^{-1}z)+\vartheta(z^{-2})\vartheta(vza^{-2})\vartheta(v^{-1}a^{-1})}{\vartheta(v^{-1}a)\vartheta(vz)}\\
&-h^{[1,1]}(v)h^{[2]}_2(v)\vartheta_0(v)\vartheta(v^{-2})\frac{\vartheta(a^{-2})\vartheta(vz^2a^{-1})\vartheta(v^{-1}z)+\vartheta(z^{-2})\vartheta(vza^{-2})\vartheta(v^{-1}a^{-1})}{\vartheta(v^{-1}a)\vartheta(vz)}
\end{align*} 
Therefore, the following theta function identities imply $h^{[2]}(v)=h^{[1,1]}(v)$ and hence $h^{[2]}_i(v)=h^{[1,1]}_i(v)$ by (\ref{Eqn_h}).

\begin{lem}
For $\epsilon=0,1$, we have
\begin{align*}
	&\vartheta(v^{-1}a)\vartheta(vz)\vartheta(za)\left(\vartheta(v^2z^{-1}a)\vartheta_{\epsilon}(vza^{-1})+\vartheta(v^2za^{-1})\vartheta_{\epsilon}(vz^{-1}a)\right)\\
	&\hspace{1em}=\left(\vartheta(a^{-2})\vartheta(vz^2a^{-1})\vartheta(v^{-1}z)+\vartheta(z^{-2})\vartheta(vza^{-2})\vartheta(v^{-1}a^{-1})\right)\vartheta(v^{-2})\vartheta_{\epsilon}(v).
\end{align*}	
\end{lem}

\begin{proof}
By using the Jacobi triple product identity, we calculate
\begin{align*}
&q^{\frac{1}{2}}(q;q)^4_{\infty}\vartheta(a^{-2})\vartheta(vz^2a^{-1})\vartheta(v^{-1}z)\vartheta(v^{-2})\vartheta_{\epsilon}(v)\\
&=\sum_{\substack{m_i\in\bb{Z},i=1,\ldots,4\\l\in\bb{Z}+\frac{\epsilon}{2}}}(-1)^{\sum_i m_i}q^{\sum_i \frac{1}{2}\left(m_i+\frac{1}{2}\right)^2+l^2}v^{m_2-m_3-2m_4-1+2l}z^{2m_2+m_3+\frac{3}{2}}a^{-2m_1-m_2-\frac{3}{2}}\\
&=\sum_{A,B\in\bb{Z}}(-1)^{A-B}z^{A+\frac{1}{2}}a^{-B-\frac{1}{2}}\sum_{\substack{d\in\frac{1}{3}\bb{Z}+\frac{\epsilon-1}{6}\\c\in\bb{Z}+2d\\b\in\bb{Z}+\frac{A-3B+3}{6}}}(-1)^{[b]+[c]}v^{-6b+6d+1}q^{f_{A,B}(b,c,d)},
\end{align*}
where we set $A=2m_2+m_3+1$, $B=2m_1+m_2+1$, and then set $b=m_1+\frac{A-3B+3}{6}$, $c=\frac{2l+m_4-1}{3}$, $d=\frac{l-m_4}{3}-\frac{1}{6}$, i.e., $m_1=b-\frac{A-3B+3}{3}$, $m_4=c-2d$, $l=c+d+\frac{1}{2}$, and  
\begin{align*}
	f_{A,B}(b,c,d)\coloneqq \frac{21}{2}b^2+\left(\frac{A+B}{2}-3\right)b+\frac{A^2}{8}-\frac{AB}{12}+\frac{B^2}{8}+\frac{1}{4}+\frac{3}{2}\left(c+\frac{1}{2}\right)^2+3d^2.
\end{align*}
By abusing notation, we also set $[x]=x-\lambda$ if we write $x\in\bb{Z}+\lambda$ in the sum, e.g., $[b]=b-\frac{A-3B+3}{3}$ and $[c]=c-2d$ in the equality above. By exchanging $a$ and $z^{-1}$, we obtain 
\begin{align*}
&-q^{\frac{1}{2}}(q;q)^4_{\infty}\vartheta(z^{-2})\vartheta(vza^{-2})\vartheta(v^{-1}a^{-1})\vartheta(v^{-2})\vartheta_{\epsilon}(v)\\
&=\sum_{A,B\in\bb{Z}}(-1)^{B-A}z^{A+\frac{1}{2}}a^{-B-\frac{1}{2}}\sum_{\substack{d\in\frac{1}{3}\bb{Z}+\frac{\epsilon-1}{6}\\c\in\bb{Z}+2d\\b\in\bb{Z}+\frac{B-3A+3}{6}}}(-1)^{[b]+[c]}v^{-6b+6d+1}q^{f_{A,B}(b,c,d)}.
\end{align*}
On the other hand, we calculate
\begin{align*}
&q^{\frac{1}{2}}(q;q)^4_{\infty}	\vartheta(v^{-1}a)\vartheta(vz)\vartheta(za)\vartheta(v^{2}z^{-1}a)\vartheta_{\epsilon}(vza^{-1})\\
&=\hspace{-2em}\sum_{\substack{m_i\in\bb{Z},i=1,\ldots,4\\l\in\bb{Z}+\frac{\epsilon}{2}}}\hspace{-1.5em}(-1)^{\sum_i m_i}q^{\sum_i \frac{1}{2}\left(m_i+\frac{1}{2}\right)^2+l^2}v^{-m_1+m_2+2m_4+2l+1}z^{m_2+m_3-m_4+2l+\frac{1}{2}}a^{m_1+m_3+m_4-2l+\frac{3}{2}}\\
&=\sum_{A,B\in\bb{Z}}(-1)^{A-B}z^{A+\frac{1}{2}}a^{-B-\frac{1}{2}}\sum_{\substack{d\in\frac{1}{3}\bb{Z}+\frac{\epsilon+1}{6}\\c\in\bb{Z}-\frac{A-B}{3}\\b\in \bb{Z}+2d-\frac{A+B+3}{6}}}(-1)^{[b]+[c]}v^{-6b+6d+1}q^{f_{A,B}(b,c,d)},
\end{align*}
where we set $A=m_2+m_3-m_4+2l$, $B=-m_1-m_3-m_4+2l-2$, and then set $b=\frac{2l-m_4}{3}-\frac{A+B+1}{6}$, $c=m_3-\frac{A-B}{3}$, and $d=\frac{l+m_4}{3}+\frac{1}{6}$, i.e., $m_3=c+\frac{A-B}{3}$, $m_4=-b+2d-\frac{A+B+3}{6}$, and $l=b+d+\frac{A+B}{6}$. Similarly, we have 
\begin{align*}
&q^{\frac{1}{2}}(q;q)^4_{\infty}	\vartheta(v^{-1}a)\vartheta(vz)\vartheta(za)\vartheta(v^{2}za^{-1})\vartheta_{\epsilon}(vz^{-1}a)\\
&=\hspace{-2em}\sum_{\substack{m_i\in\bb{Z},i=1,\ldots,4\\l\in\bb{Z}+\frac{\epsilon}{2}}}\hspace{-1.5em}(-1)^{\sum_i m_i}q^{\sum_i \frac{1}{2}\left(m_i+\frac{1}{2}\right)^2+l^2}v^{-m_1+m_2+2m_4+2l+1}z^{m_2+m_3+m_4-2l+\frac{3}{2}}a^{m_1+m_3-m_4+2l+\frac{1}{2}}\\
&=\sum_{A,B\in\bb{Z}}(-1)^{A-B}z^{A+\frac{1}{2}}a^{-B-\frac{1}{2}}\sum_{\substack{d\in\frac{1}{3}\bb{Z}+\frac{\epsilon+1}{6}\\c\in\bb{Z}-\frac{A-B}{3}\\b\in\bb{Z}-2d-\frac{A+B-3}{6}}}(-1)^{[b]+[c]}v^{-6b+6d+1}q^{f_{A,B}(b,c,d)},
\end{align*}
where we set $A=m_2+m_3+m_4-2l+1$, $B=-m_1-m_3+m_4-2l-1$, and then set $b=\frac{-2l+m_4}{3}-\frac{A+B-1}{6}$, $c=m_3-\frac{A-B}{3}$, and $d=\frac{l+m_4}{3}+\frac{1}{6}$, i.e., $m_3=c+\frac{A-B}{3}$, $m_4=b+2d+\frac{A+B-3}{6}$, and $l=-b+d-\frac{A+B}{6}$.

Hence it is enough to check
\begin{align*}
&\sum_{\substack{b\in\bb{Z}-\frac{A+B-3}{6}+\frac{A-B}{3}\\c\in\bb{Z}+2d}}(-1)^{[b]+[c]}v^{-6b}q^{f_{A,B}(b,c,d)}-\sum_{\substack{b\in\bb{Z}-\frac{A+B-3}{6}+\frac{B-A}{3}\\c\in\bb{Z}+2d}}(-1)^{[b]+[c]}v^{-6b}q^{f_{A,B}(b,c,d)}\\
&=-\sum_{\substack{b\in\bb{Z}+2d-\frac{A+B-3}{6}\\c\in\bb{Z}-\frac{A-B}{3}}}(-1)^{[b]+[c]}v^{-6b}q^{f_{A,B}(b,c,d)}+\sum_{\substack{b\in\bb{Z}-2d-\frac{A+B-3}{6}\\c\in\bb{Z}-\frac{A-B}{3}}}(-1)^{[b]+[c]}v^{-6b}q^{f_{A,B}(b,c,d)}
\end{align*}
for any $d\in\frac{1}{6}\bb{Z}$. We note that by $f_{A,B}(b,-1-c,d)=f_{A,B}(b,c,d)$, we have
\begin{align}\label{eq_cancelation}
	\sum_{c\in\bb{Z}+\lambda}(-1)^{[c]}q^{f_{A,B}(b,c,d)}=-\sum_{c\in\bb{Z}-\lambda}(-1)^{[c]}q^{f_{A,B}(b,c,d)}
\end{align}

If $2d\in\bb{Z}$, then the LHS is zero by (\ref{eq_cancelation}) for $\lambda=0$ and the RHS cancels to zero. If $\frac{A-B}{3}\in\bb{Z}$, then the LHS cancels to zero and the RHS is zero by (\ref{eq_cancelation}) for $\lambda=0$. 

If $2d-\frac{A-B}{3}\in\bb{Z}$, then the first term of the LHS coincides with the first term of the RHS by (\ref{eq_cancelation}) and the second term of the LHS coincides with the second term of the RHS by (\ref{eq_cancelation}). 

If $2d+\frac{A-B}{3}\in\bb{Z}$, then the first term of the LHS coincides with the second term of the RHS and the second term of the LHS coincides with the first term of the RHS. 

Since $2d,\frac{A-B}{3}\in\frac{1}{3}\bb{Z}$, at least one of the above conditions is satisfied. This completes the proof of the above equality.
\end{proof}

In summary, we obtained the formula for $\mca{E}^{ell}_{\mf{X}}([2])$ and $\mca{E}^{ell}_{\mf{X}}([1,1])$ in Theorem~\ref{Main_Thm} by setting $f_0(v;q)\coloneqq h^{[2]}(v)=h^{[1,1]}(v)$, $f_1(v;q)\coloneqq -h^{[2]}_{2}(v)=-h^{[1,1]}_{2}(v)$, and $f_2(v;q)\coloneqq h^{[2]}_{0}(v)=h^{[1,1]}_{0}(v)$.

\subsection{$K$-theory limits}

In order to prove our main theorem, it remains to check the asymptotics of $f_{i}(v;q)$ at $q=0$. Let us denote the leading term of $f_i(v;q)$ by $q^{c_i}f_{i,0}(v)$. For any $s\in\bb{R}$ and $\epsilon=0,1$, the function $l\mapsto\left(l+\frac{\epsilon}{2}\right)^2-2s\left(l+\frac{\epsilon}{2}\right)$ on $\bb{Z}$ takes its minimum when 
\begin{align*}
	l=\begin{cases}
		\left\lfloor s-\frac{\epsilon-1}{2}\right\rfloor&\mbox{ if }s\notin\bb{Z}+\frac{\epsilon-1}{2},\\
		\left\lfloor s-\frac{\epsilon-1}{2}\right\rfloor,\left\lfloor s-\frac{\epsilon+1}{2}\right\rfloor &\mbox{ if }s\in\bb{Z}+\frac{\epsilon-1}{2}.
	\end{cases}
\end{align*} 
This implies that the leading terms of $\delta_z^{-s}\left(\mca{E}^{ell}_{\mf{X}}([1,1])\right)$ are given by 
\begin{align*}
	\begin{cases}
		(-1)^{\lfloor s\rfloor}q^{c_0+\frac{1}{2}\left(\lfloor s\rfloor+\frac{1}{2}\right)\left(-2s+\lfloor s\rfloor+\frac{1}{2}\right)}f_{0,0}(v)z^{\lfloor s\rfloor+\frac{1}{2}}\mca{O}\left(\lfloor s\rfloor+\frac{1}{2}\right) &\mbox{ if }s\notin\bb{Z},\\
		(-1)^{s}q^{c_0-\frac{1}{2}\left(s+\frac{1}{2}\right)\left(s-\frac{1}{2}\right)}f_{0,0}(v)\left(z^{s+\frac{1}{2}}\mca{O}\left(s+\frac{1}{2}\right)-z^{s-\frac{1}{2}}\mca{O}\left(s-\frac{1}{2}\right)\right) &\mbox{ if }s\in\bb{Z}.
	\end{cases}
\end{align*}
In particular, the normalization (\ref{K-limit2}) holds if and only if $f_{0,0}(v)=1$, and if $f_{0,0}(v)=1$, then Property~A holds for $\mca{E}^{ell}_{\mf{X}}([1,1])$ for any $s\in\bb{R}$ by comparing with Proposition~\ref{prop_wall_K-th_can}. Similarly, the leading terms of the first term of $\delta_z^{-s}\left(\mca{E}^{ell}_{\mf{X}}([2])\right)$ are given by
\begin{align*}
	\begin{cases}
		(-1)^{s}f_{1,0}(v)q^{c_1-\frac{3}{2}s^2+\frac{1}{8}}\left(v^{}z^{3s+\frac{1}{2}}\mca{O}\left(s-\frac{1}{2}\right)-v^{-1}z^{3s-\frac{1}{2}}\mca{O}\left(s+\frac{1}{2}\right)\right) &\mbox{ if }s\in\bb{Z},\\
		(-1)^{\lfloor s\rfloor}f_{1,0}(v)q^{c_1+\frac{3}{2}\lfloor s\rfloor^2-3s\lfloor s\rfloor+\frac{\lfloor s\rfloor}{2}-\frac{s}{2}+\frac{1}{8}}vz^{3\lfloor s\rfloor+\frac{1}{2}}\mca{O}\left(\lfloor s\rfloor-\frac{1}{2}\right) &\mbox{ if }s<\lfloor s\rfloor+\frac{1}{2},\\
		(-1)^{s-\frac{1}{2}}f_{1,0}(v)q^{c_1-\frac{3}{2}s^2+\frac{1}{4}}\left(v^{-1}z^{3s+1}\mca{O}\left(s+1\right)+vz^{3s-1}\mca{O}\left(s-1\right)\right) &\mbox{ if }s\in\bb{Z}+\frac{1}{2},\\
		(-1)^{\lfloor s\rfloor}f_{1,0}(v)q^{c_1+\frac{3}{2}\lfloor s\rfloor^2-3s\lfloor s\rfloor+\frac{5}{2}\lfloor s\rfloor-\frac{5}{2}s+\frac{9}{8}}v^{-1}z^{3\lfloor s\rfloor+\frac{3}{2}}\mca{O}\left(\lfloor s\rfloor+\frac{3}{2}\right) &\mbox{ if }s>\lfloor s\rfloor+\frac{1}{2},
	\end{cases}
\end{align*}
and the leading terms of the second term of $\delta_z^{-s}\left(\mca{E}^{ell}_{\mf{X}}([2])\right)$ are given by
\begin{align*}
	\begin{cases}
		(-1)^{\lfloor s\rfloor}f_{2,0}(v)q^{c_2+\frac{3}{2}\lfloor s\rfloor^2-3s\lfloor s\rfloor+\frac{3}{2}\lfloor s\rfloor-\frac{3}{2}s+\frac{3}{8}}z^{3\lfloor s\rfloor+\frac{3}{2}}\mca{O}\left(\lfloor s\rfloor+\frac{1}{2}\right) &\mbox{ if }s\notin\bb{Z},\\
		(-1)^{s}f_{2,0}(v)q^{c_2-\frac{3}{2}s^2+\frac{3}{8}}\left(z^{3s+\frac{3}{2}}\mca{O}\left(s+\frac{1}{2}\right)-v^{-2}z^{3s+\frac{1}{2}}\mca{O}\left(s+\frac{3}{2}\right)\right.\\\hspace{10em}\left.+v^2z^{3s-\frac{1}{2}}\mca{O}\left(s-\frac{3}{2}\right)-z^{3s-\frac{3}{2}}\mca{O}\left(s-\frac{1}{2}\right)\right) &\mbox{ if }s\in\bb{Z}.
	\end{cases}
\end{align*}
In particular, the normalization (\ref{K-limit1}) holds if and only if $f_{1,0}(v)=1$ and $c_1<c_2+\lfloor s\rfloor-s+\frac{1}{4}$ for any $0<s<\frac{1}{2}$.

If $s\in\bb{Z}$, then in order for the leading term to be the $K$-theoretic canonical bases, we need $c_1< c_2+\frac{1}{4}$. If $s\in\bb{Z}+\frac{1}{2}$, then we need $c_1<c_2-\frac{1}{4}$. If $s<\lfloor s\rfloor+\frac{1}{2}$, then we need $c_1<c_2+\lfloor s\rfloor-s+\frac{1}{4}$, i.e., $c_1\leq c_2-\frac{1}{4}$. If $s>\lfloor s\rfloor+\frac{1}{2}$, then we need $c_1<c_2-\lfloor s\rfloor+s-\frac{3}{4}$, i.e., $c_1\leq c_2-\frac{1}{4}$. Therefore, we need $c_2>c_1+\frac{1}{4}$. 

On the other hand, since $q^{-c_1-\frac{1}{8}}\mca{E}^{ell}_{\mf{X}}([2])$ should contain only half-integral powers of $q$, we should have $c_2-c_1+\frac{1}{4}\in\frac{1}{2}\bb{Z}$. Therefore we obtain $c_2\geq c_1+\frac{3}{4}$. This completes the proof of Theorem~\ref{Main_Thm}.

\subsection{Elliptic bar invariance}

In this section, we check that our formula is compatible with Property~C and hence prove the elliptic bar invariance of $\mca{E}^{ell}_{\mf{X}}([2])$ and $\mca{E}^{ell}_{\mf{X}}([1,1])$ under further assumption that $f_i(v^{-1};q)=f_i(v;q)$ for $i=0,1,2$. We set $\mca{E}^{ell}_{-\mf{X}}\coloneqq\mca{E}^{ell}_{\mf{X}}$ and $\mca{E}^{ell}_{\mf{X}^!_{\mr{flop}}}\coloneqq\overline{\mca{E}^{ell}_{\mf{X}^!}}$. We note that for the dual pair $(-\mf{X},\mf{X}^!_{\mr{flop}})$, we have $\sigma_{-\mf{X},\mf{X}^!_{\mr{flop}}}(p)=-1$ for any $p\in X^{H}$.

\begin{cor}
Assume $f_i(v^{-1};q)=f_i(v;q)$ for $i=0,1,2$. Then we have
\begin{align*}
	-\Upsilon(v;q)\cdot\mr{Stab}^{ell}_{-\mf{X}}=\mca{E}^{ell}_{-\mf{X}}\cdot^t\!\mca{E}^{ell}_{\mf{X}^!_{\mr{flop}}}.
\end{align*}
In particular, we have $\beta^{ell}_{\mf{X}}\left(\mca{E}^{ell}_{\mf{X}}(\mu)\right)=\mca{E}^{ell}_{\mf{X}}(\mu)$ for $\mu=[2],[1,1]$.
\end{cor}

\begin{proof}
It is straightforward to check from the explicit formula and the assumption $f_i(v^{-1};q)=f_i(v;q)$ that we have
\begin{align*}
\left.\mca{E}^{ell}_{\mf{X}}\right|_{a\mapsto a^{-1}}&=\Omega\cdot\mca{E}^{ell}_{\mf{X}},\\
\overline{\mca{E}^{ell}_{\mf{X}}}&=-\left.\mca{E}^{ell}_{\mf{X}}\right|_{a\mapsto a^{-1},z\mapsto z^{-1}},
\end{align*}
where we set 
\begin{align*}
	\Omega\coloneqq\left(\begin{array}{cc}
		0 & 1\\
		1 & 0
	\end{array}\right).
\end{align*}
By using $\mca{E}^{ell}_{\mf{X}^!}=\Omega\cdot\left.\mca{E}^{ell}_{\mf{X}}\right|_{a\leftrightarrow z}\cdot\Omega$, we have
\begin{align*}
	\mca{E}^{ell}_{\mf{X}^!_{\mr{flop}}}=\Omega\cdot\left.\overline{\mca{E}^{ell}_{\mf{X}}}\right|_{a\leftrightarrow z}\cdot\Omega=-\Omega\cdot\left.\left(\left.\mca{E}^{ell}_{\mf{X}}\right|_{a\mapsto a^{-1}}\right)\right|_{a\mapsto z,z\mapsto a^{-1}}\cdot\Omega=-\left.\mca{E}^{ell}_{\mf{X}}\right|_{a\mapsto z,z\mapsto a^{-1}}\cdot\Omega.
\end{align*}
Hence by using (\ref{ell_Unstab}) and (\ref{duality}), we obtain
\begin{align*}
	-\Upsilon(v;q)\cdot\mr{Stab}^{ell}_{-\mf{X}}&=-\Upsilon(v;q)\cdot\Omega\cdot\left.\mr{Stab}^{ell}_{\mf{X}}\right|_{a\mapsto a^{-1}}\cdot\Omega\\
	&=-\Omega\cdot\left.\mca{E}^{ell}_{\mf{X}}\right|_{a\mapsto a^{-1}}\cdot\Omega\cdot\left.^t\mca{E}^{ell}_{\mf{X}}\right|_{a\mapsto z,z\mapsto a^{-1}}\\
	&=\mca{E}^{ell}_{-\mf{X}}\cdot^t\!\mca{E}^{ell}_{\mf{X}^!_{\mr{flop}}}.
\end{align*}
By Definition~\ref{Dfn_ellbar}, we have $\beta^{ell}_{\mf{X}}\left(\mr{Stab}^{ell}_{\mf{X}}(p)\right)=-\mr{Stab}^{ell}_{-\mf{X}}(p)$. Therefore, we obtain 
\begin{align*}
	\sum_{\mu\in\Xi_{\mf{X}}}\overline{\left.\mca{E}^{ell}_{\mf{X}^!}(\mu)\right|_{p^!}}\cdot\beta^{ell}_{\mf{X}}\left(\mca{E}^{ell}_{\mf{X}}(\mu)\right)&=\Upsilon(v;q)\cdot\beta^{ell}_{\mf{X}}\left(\mr{Stab}^{ell}_{\mf{X}}(p)\right)\\
&=\sum_{\mu\in\Xi_{\mf{X}}}\overline{\left.\mca{E}^{ell}_{\mf{X}^!}(\mu)\right|_{p^!}}\cdot\mca{E}^{ell}_{\mf{X}}(\mu)
\end{align*}
for any $p\in X^H$. Since the matrix $\mca{E}^{ell}_{\mf{X}^!}$ is invertible, we obtain $\beta^{ell}_{\mf{X}}\left(\mca{E}^{ell}_{\mf{X}}(\mu)\right)=\mca{E}^{ell}_{\mf{X}}(\mu)$.
\end{proof}

\subsection{Solving $q$-difference equations for $v$}

In this section, we further restrict $f_i(v;q)$ by requiring Property~E. By straightforward calculations, we obtain
\begin{align*}
	\delta_v\left(\mca{E}^{ell}_{\mf{X}}([1,1])\right)&=\frac{\delta_{v}(f_0)}{f_0}\cdot q^{-2}z^{-2}\mca{O}(-2)\otimes\mca{E}^{ell}_{\mf{X}}([1,1]),\\
	\delta_v\left(\mca{E}^{ell}_{\mf{X}}([2])\right)&=q^{-1}v^2z^{-2}\mca{O}(-2)\otimes\left(\delta_v(f_1)\cdot \mca{E}_0+\delta_v(f_2)\cdot\mca{E}_1\right),
\end{align*}
where we set 
\begin{align*}
	\mca{E}_i\coloneqq\sum_{l,m\in\bb{Z}}(-1)^{m}q^{\left(l+\frac{i}{2}\right)^2+\frac{1}{2}\left(m+\frac{1}{2}\right)^2}v^{-2l-i+2m+1}z^{2l+i+m+\frac{1}{2}}\mca{O}\left(2l+i-m-\frac{1}{2}\right).
\end{align*}
In order for them to satisfy the same $q$-difference equations, it is enough to assume 
\begin{align*}
	\delta_v\left(\frac{f_i(v;q)}{f_0(v;q)}\right)=q^{-1}v^{-2}\cdot\frac{f_i(v;q)}{f_0(v;q)}
\end{align*}
for $i=1,2$. In particular, if $f_0(v;q)$ does not depend on $v$, then $f_i(v;q)$ for $i=1,2$ should be a linear combination of $\vartheta_0(v)$ and $\vartheta_1(v)$ defined in (\ref{Eqn_theta}). 

\subsection{Speculations}

We finish this paper by commenting possible relations to the theory of vertex operator superalgebras. Another expected property of the elliptic canonical bases which is not pursued further in this paper is that their specializations are given by the supercharacters of simple modules of some VOSA. In the toric cases, the corresponding VOSA is the lattice VOSA (\cite{H1}) and contained in the boundary VOSA in the sense of Costello-Gaiotto \cite{CG} for the corresponding abelian 3d $\mca{N}=4$ gauge theories in the B twists by the work of Ballin-Creutzig-Dimofte-Niu \cite[Proposition 3.2]{BCDN}. We do not know the corresponding VOSA in the case of Hilbert scheme of 2-points, but should be related to the recent work of Arakawa-Kuwabara-M\"oller \cite{AKM}. If one could find the corresponding VOSA, then the normalization would be fixed.


\end{document}